\documentclass[english]{article}
\usepackage[T1]{fontenc}
\usepackage[latin9]{inputenc}
\usepackage{geometry}
\usepackage{babel}
\usepackage{array}
\usepackage{refstyle}
\usepackage{wrapfig}
\usepackage{units}
\usepackage{multirow}
\usepackage{amsthm}
\usepackage{amsmath}
\usepackage{amssymb}
\usepackage{graphicx}
\usepackage{esint}
\usepackage[all]{xy}
\usepackage[unicode=true,pdfusetitle,
 bookmarks=true,bookmarksnumbered=false,bookmarksopen=false,
 breaklinks=false,pdfborder={0 0 1},backref=false,colorlinks=false]
 {hyperref}

\makeatletter

\AtBeginDocument{\providecommand\secref[1]{\ref{sec:#1}}}
\AtBeginDocument{\providecommand\corref[1]{\ref{cor:#1}}}
\AtBeginDocument{\providecommand\thmref[1]{\ref{thm:#1}}}
\AtBeginDocument{\providecommand\claref[1]{\ref{cla:#1}}}
\AtBeginDocument{\providecommand\enuref[1]{\ref{enu:#1}}}
\AtBeginDocument{\providecommand\tabref[1]{\ref{tab:#1}}}
\AtBeginDocument{\providecommand\figref[1]{\ref{fig:#1}}}
\AtBeginDocument{\providecommand\defref[1]{\ref{def:#1}}}
\AtBeginDocument{\providecommand\propref[1]{\ref{prop:#1}}}
\AtBeginDocument{\providecommand\eqref[1]{\ref{eq:#1}}}
\AtBeginDocument{\providecommand\lemref[1]{\ref{lem:#1}}}
\AtBeginDocument{\providecommand\subref[1]{\ref{sub:#1}}}
\providecommand{\tabularnewline}{\\}
\RS@ifundefined{subref}
  {\def\RSsubtxt{section~}\newref{sub}{name = \RSsubtxt}}
  {}
\RS@ifundefined{thmref}
  {\def\RSthmtxt{theorem~}\newref{thm}{name = \RSthmtxt}}
  {}
\RS@ifundefined{lemref}
  {\def\RSlemtxt{lemma~}\newref{lem}{name = \RSlemtxt}}
  {}

\numberwithin{equation}{section}
\numberwithin{figure}{section}
\theoremstyle{plain}
\newtheorem{thm}{\protect\theoremname}[section]
  \theoremstyle{definition}
  \newtheorem{defn}[thm]{\protect\definitionname}
  \theoremstyle{plain}
  \newtheorem{cor}[thm]{\protect\corollaryname}
  \theoremstyle{plain}
  \newtheorem*{thm*}{\protect\theoremname}
  \theoremstyle{plain}
  \newtheorem{prop}[thm]{\protect\propositionname}
  \theoremstyle{remark}
  \newtheorem{claim}[thm]{\protect\claimname}
  \theoremstyle{remark}
  \newtheorem{rem}[thm]{\protect\remarkname}
  \theoremstyle{plain}
  \newtheorem{lem}[thm]{\protect\lemmaname}

\usepackage[perpage,symbol]{footmisc}
\usepackage[leftcaption]{sidecap}
\usepackage{doi}

\sidecaptionvpos{figure}{c}

\newcommand{\FigBesBeg}[1][1.0]{%
 \let\MyFigure\figure
 \let\MyEndfigure\endfigure
 \renewenvironment{figure}[1]{\begin{SCfigure}[#1]##1}{\end{SCfigure}}}

\newcommand{\FigBesEnd}{%
 \let\figure\MyFigure
 \let\endfigure\MyEndfigure}

\DefineFNsymbols*{lamportnostar}[math]{\dagger\ddagger\S\P\|{\dagger\dagger}{\ddagger\ddagger}}
\setfnsymbol{lamportnostar}

\newcommand{\F}{\mathbf{F}}
\newcommand{\hF}{{\widehat{\F}_k}}

\renewcommand{\G}{\Gamma}

\newcommand{\ff}{\mathrel{\smash{\scalebox{1}[0.7]{\ensuremath{\stackrel{*}{\le}}}}}}
\newcommand{\pff}{\mathrel{\smash{\scalebox{1}[0.7]{\ensuremath{\stackrel{*}{\lneqq}}}}}}
\newcommand{\fg}{\le_{f\!g}}
\newcommand{\D}{\mathcal{D}}
\newcommand{\la}{\langle}
\newcommand{\ra}{\rangle}

\renewcommand{\th}{\twoheadrightarrow}

\newcommand{\rp}{\widetilde{\pi}}
\newcommand{\alg}{\le_{alg}}
\newcommand{\algne}{\lneqq_{alg}}
\newcommand{\cJ}{\widehat{\Gamma}_{X}\left(J\right)}
\newcommand{\Xcov}{\scriptscriptstyle{\stackrel{\twoheadrightarrow}{X}}}
\newcommand{\covers}{\mathrel{\leq_{\smash{\scalebox{0.9}[0.8]{\ensuremath{\Xcov}}}}}}
\newcommand{\ncovers}{\mathrel{\lneqq_{\smash{\scalebox{0.9}[0.8]{\ensuremath{\Xcov}}}}}}

\newcommand{\XC}[2]{\left[#1,#2\right]_{\Xcov}}
\newcommand{\XCO}[2]{\left[#1,#2\right)_{\Xcov}}
\newcommand{\XF}[1]{\XCO{#1}{\infty}}
\renewcommand{\AE}[1]{\left[#1,\infty\right)_{alg}}

\DeclareMathOperator{\rk}{rk}
\DeclareMathOperator{\rrk}{\hspace{-0.5pt}\widetilde{\hspace{0.5pt}rk\hspace{0.5pt}}\hspace{-0.5pt}}
\DeclareMathOperator{\Hom}{Hom}
\DeclareMathOperator{\Sym}{Sym}
\DeclareMathOperator{\crit}{Crit}

\DeclareMathOperator{\Aut}{Aut}

\usepackage[all]{xy}

\newref{eq}{refcmd={(\ref{#1})}}
\newref{enu}{refcmd={(\ref{#1})}}
\newref{thm}{refcmd={\ref{#1}}}
\newref{lem}{refcmd={\ref{#1}}}
\newref{sec}{refcmd={\ref{#1}}}
\newref{sub}{refcmd={\ref{#1}}}
\newref{tab}{refcmd={\ref{#1}}}
\newref{fig}{refcmd={\ref{#1}}}

\theoremstyle{plain}
\newtheorem{claim}[thm]{\protect\claimname}
\newtheorem{obs}[thm]{Observation}
\makeatother

  \providecommand{\claimname}{Claim}
  \providecommand{\corollaryname}{Corollary}
  \providecommand{\definitionname}{Definition}
  \providecommand{\lemmaname}{Lemma}
  \providecommand{\propositionname}{Proposition}
  \providecommand{\remarkname}{Remark}
  \providecommand{\theoremname}{Theorem}
\providecommand{\theoremname}{Theorem}

\begin{document}

\title{Measure Preserving Words are Primitive}

\author{Doron Puder%
\thanks{Supported by the ERC and by Adams Fellowship Program of the Israel
Academy of Sciences and Humanities.%
} \texttt{\small ~~~~~~~} Ori Parzanchevski%
\thanks{Supported by Advanced ERC Grant.%
}
\\
{\small }\\
{\small Einstein Institute of Mathematics }\\
{\small{} Hebrew University, Jerusalem}\\
{\small{} }\texttt{\small doronpuder@gmail.com~~~~~~~parzan@math.huji.ac.il}}
\maketitle
\begin{abstract}
We establish new characterizations of primitive elements and free
factors in free groups, which are based on the distributions they
induce on finite groups. For every finite group $G$, a word $w$
in the free group on $k$ generators induces a \emph{word map} from
$G^{k}$ to $G$. We say that $w$ is measure preserving with respect
to $G$ if given uniform distribution on $G^{k}$, the image of this
word map distributes uniformly on $G$. It is easy to see that primitive
words (words which belong to some basis of the free group) are measure
preserving w.r.t.\ all finite groups, and several authors have conjectured
that the two properties are, in fact, equivalent. Here we prove this
conjecture. The main ingredients of the proof include random coverings
of Stallings graphs, algebraic extensions of free groups, and Möbius
inversions. Our methods yield the stronger result that a subgroup
of $\F_{k}$ is measure preserving if and only if it is a free factor.

As an interesting corollary of this result we resolve a question on
the profinite topology of free groups and show that the primitive
elements of $\F_{k}$ form a closed set in this topology. 
\end{abstract}
\tableofcontents{}

\section{\label{sec:Introduction}Introduction}

This paper establishes a new characterization of primitive elements
in free groups, which is based on the distributions they induce on
finite groups. Let $\F_{k}$ be the free group on $k$ generators
$X=\left\{ x_{1,}\ldots,x_{k}\right\} $, and let $w=\prod_{j=1}^{r}x_{i_{j}}^{\varepsilon_{j}}$
($\varepsilon_{j}=\pm1$) be a word in $\mathbf{F}_{k}$. For every
group $G$, $w$ induces a \emph{word map} from the Cartesian product
$G^{k}$ to $G$, by substitutions:
\[
w:\left(g_{1},\ldots,g_{k}\right)\mapsto\prod_{j=1}^{r}g_{i_{j}}^{\varepsilon j}.
\]
The word $w$ is called \emph{measure preserving} with respect to
a finite group $G$ if all the fibers of this map are of equal size.
Namely, every element in $G$ is obtained by substitutions in $w$
the same number of times. We say that $w$ is \emph{measure preserving}
if it is measure preserving w.r.t.\ every finite group. The last
years have seen a great interest in word maps in groups, and the distributions
they induce. We refer the reader, for instance, to \cite{Sha09,LSh09,AV10,parzanchevski2012fourier},
and to the recent book \cite{Seg09} and survey \cite{Shalev2013}.
Several authors have also studied words which are asymptotically measure
preserving on finite simple groups, see e.g.\ \cite{LSh08,GSh09,BK12}.
\medskip{}

The word $w$ is called \emph{primitive} if it belongs to some basis
(free generating set) of $\mathbf{F}_{k}$. It is a simple observation
(see \ref{obs:prim-is-mp} below) that primitive words are measure
preserving, and several authors have conjectured that the converse
is also true. Namely, that measure preservation implies primitivity%
\footnote{It is interesting to note that there is an easy abelian parallel to
this conjecture. A word $w\in\F_{k}$ belongs to a basis of $\mathbb{Z}^{k}\cong\F_{k}/\F'_{k}$
iff for any group $G$ the associated word map is surjective. See
\cite{Seg09}, Lemma 3.1.1.%
}. From private conversations we know that this has occurred to the
following mathematicians and discussed among themselves: N. Avni,
T. Gelander, M. Larsen, A. Lubotzky and A. Shalev. The question was
independently raised in \cite{LP10} and also in \cite{AV10}, alongside
a generalization of it (see Section \ref{sec:Profinite}).

In \cite{Pud14a} the first author proved the conjecture for $\F_{2}$.
Here we prove it in full:
\begin{thm}
\label{thm:mp-word-is-prim}A measure preserving word is primitive.
\end{thm}
A key ingredient of the proof is the extension of the problem from
single words to (finitely generated) subgroups of $\F_{k}$. The concept
of primitive words extends naturally to the notion of free factors:
Let $H$ be a subgroup of the free group $J$ (in particular, $H$
is free as well). We say that $H$ is a \emph{free factor} of $J$,
and denote this by $H\ff J$, if there is a subgroup $H'\le J$ such
that $H*H'=J$. Equivalently, $H\ff J$ iff some basis of $H$ can
be extended to a basis of $J$. (This in turn is easily seen to be
equivalent to the condition that \emph{every} basis of $H$ extends
to a basis of $J$.)

In order to generalize the notion of measure preservation to subgroups,
we need to change a little our perspective of word maps. One can think
of the word map $w$ as the evaluation map from $\Hom\left(\F_{k},G\right)$
to $G$, i.e., $w\left(\alpha\right)=\alpha\left(w\right)$ for $\alpha\in\Hom\left(\F_{k},G\right)$.
The identification of $\Hom\left(\F_{k},G\right)$ with $G^{k}$ depends
on the chosen basis, and is due to the fact that a homomorphism from
a free group is uniquely determined by choosing the images of the
elements of a basis, and these images can be chosen arbitrarily. 

In this perspective, $w$ is measure preserving w.r.t.\ $G$ if the
element $\alpha_{G}\left(w\right)$ is uniformly distributed over
$G$, where $\alpha_{G}\in\Hom\left(\F_{k},G\right)$ is a homomorphism
chosen uniformly at random. If $w$ is primitive then it belongs to
some basis, and identifying $\Hom\left(\F_{k},G\right)$ and $G^{k}$
according to this basis gives

\begin{obs}\label{obs:prim-is-mp}A primitive word is measure preserving.\end{obs}

We can now extend the notion of measure preservation from words to
finitely generated subgroups (we write $H\fg\F_{k}$ when $H$ is
a finitely generated subgroups of $\F_{k}$):
\begin{defn}
\label{def:mp_subgroup}Let $H\fg\F_{k}$. We say that $H$ is \emph{measure
preserving} if for every finite group $G$ and $\alpha_{G}\in\Hom\left(\F_{k},G\right)$
a random homomorphism chosen with uniform distribution, $\alpha_{G}\big|_{H}$
is uniformly distributed in $\Hom\left(H,G\right)$. 
\end{defn}
This can be reformulated in terms of distributions of subgroups: Observe
the distribution of the random subgroup $\alpha_{G}\left(H\right)\leq G$,
where $\alpha_{G}\in\Hom\left(\F_{k},G\right)$ distributes uniformly.
Then $H$ is measure preserving if the distribution of $\alpha_{G}\left(H\right)$
is the same as that of the image of a uniformly chosen homomorphism
from $\mathbf{F}_{\rk\left(H\right)}$ to $G$ (where $\rk\left(H\right)$
denotes the rank of $H$). 

As for single words, it is immediate that a free factor is measure
preserving, and again it is natural to conjecture that the converse
also holds. Since $1\ne w\in\F_{k}$ is measure preserving iff $\left\langle w\right\rangle $
is measure preserving, this is an extension of the conjecture regarding
words. In \cite{Pud14a} the first author proved the extended conjecture
for subgroups of $\F_{k}$ of rank $\ge k-1$ (thus proving the conjecture
for $\F_{2}$), but the techniques used in that paper are specialized
for the proven cases. In this paper we introduce completely new techniques,
which yield the extended conjecture in full:
\begin{thm}
\label{thm:mp_is_prim}A measure preserving subgroup is a free factor.
\end{thm}
In Section \secref{Profinite} we explain how this circle of ideas
is related to the study of profinite groups and decidability questions.
In fact, part of the original motivation for this study comes from
this relation. In particular we have the following corollary (see
also Corollary \corref{prim_in_F_k_iff_prim_in_hF_k}):
\begin{cor}
\label{cor:prim_are_closed}The set $P$ of primitive elements in
$\F_{k}$ is closed in the profinite topology. 
\end{cor}
In plain terms, this amounts to the assertion that every non-primitive
word in $\F_{k}$ is contained in a primitive-free coset of a finite
index subgroup. \medskip{}

In order to prove Theorem \thmref{mp_is_prim}, one needs to exhibit,
for each non-primitive word $w\in\F_{k}$, some ``witness'' finite
group with respect to which $w$ is not measure preserving. Our witnesses
are always the symmetric groups $S_{n}.$ In fact, it is enough to
restrict one's attention to the average number of fixed points in
the random permutation $\alpha_{S_{n}}\left(w\right)$ (which we also
denote by $\alpha_{n}\left(w\right)$). We summarize this in the following
stronger version of Theorems \ref{thm:mp-word-is-prim} and \thmref{mp_is_prim}: 
\begin{thm*}[\ref*{thm:mp_is_prim}']
\label{thm:mp_is_prim_ext} Let $w\in\F_{k}$, and for every finite
group $G$, let $\alpha_{G}\in\Hom\left(\F_{k},G\right)$ denote a
random homomorphism chosen with uniform distribution. Then the following
are equivalent:
\begin{enumerate}
\item $w$ is primitive.
\item $w$ is measure preserving: for every finite group $G$ the random
element $\alpha_{G}\left(w\right)$ has uniform distribution.
\item For every $n\in\mathbb{N}$ the random permutation $\alpha_{n}\left(w\right)=\alpha_{S_{n}}\left(w\right)$
has uniform distribution.
\item For every $n\in\mathbb{N}$, the expected number of fixed points in
the random permutation $\alpha_{n}\left(w\right)=\alpha_{S_{n}}\left(w\right)$
is 1:
\[
\mathbb{E}\left[\#\mathrm{fix}\left(\alpha_{n}(w)\right)\right]=1
\]

\item For infinitely many $n\in\mathbb{N}$, 
\[
\mathbb{E}\left[\#\mathrm{fix}\left(\alpha_{n}(w)\right)\right]\le1
\]

\end{enumerate}
The analogue properties for f.g.\ subgroups are equivalent as well.
For example, the parallel of property $\left(4\right)$ for $H\fg\F_{k}$
is that for every $n$, the image $\alpha_{n}\left(H\right)\subseteq S_{n}$
stabilizes on average exactly $n^{1-\rk\left(H\right)}$ elements
of \textup{$\left\{ 1,\ldots,n\right\} $}. 
\end{thm*}
We already explained above the implication $\left(1\right)\Rightarrow\left(2\right)$,
and $\left(2\right)\Rightarrow\left(3\right)\Rightarrow\left(4\right)\Rightarrow\left(5\right)$
is evident (recall that a uniformly distributed random permutation
has exactly one fixed point on average). The only nontrivial, somewhat
surprising part, is the implication $\left(5\right)\Rightarrow\left(1\right)$
which is proven in this paper. It turns out that an effective bound
can also be obtained:
\begin{prop}
\label{prop:effective-bound}A word $w$ of length $\ell>0$ is primitive
iff $\mathbb{E}\left[\#\mathrm{fix}\left(\alpha_{n}(w)\right)\right]=1$
for $n\leq\ell$.
\end{prop}
An analogue result holds for subgroups (see Corollary \ref{cor:effective-bound-subgroups}).

\medskip{}

A key role in our proof is played by the notion of \emph{primitivity
rank}, an invariant classifying words and f.g.\ subgroups of $\F_{k}$,
which was first introduced in \cite{Pud14a}:\emph{ }A primitive word
$w\in\F_{k}$ is also primitive in every subgroup containing it (Claim
\claref{ff-properties}\enuref{intermideate-in-free}). However, if
$w$ is not primitive in $\F_{k}$, it may be either primitive or
non-primitive in subgroups of $\F_{k}$ containing it. But what is
the smallest rank of a subgroup giving evidence to the imprimitivity
of $w$? Informally, how far does one have to search in order to establish
that $w$ is not primitive? Concretely:
\begin{defn}
\label{def:prim_rank} The \emph{primitivity rank} of $w\in\F_{k}$,
denoted $\pi\left(w\right)$, is 
\[
\pi(w)=\min\left\{ \rk\left(J\right)\,\middle|\,\begin{gathered}w\in J\le\F_{k}~s.t.\\
w\textrm{ is \textbf{not} primitive in \ensuremath{J}}
\end{gathered}
\right\} .
\]
 If no such $J$ exists, i.e.\ if $w$ is primitive, then $\pi\left(w\right)=\infty$. 

More generally, for $H\fg\F_{k}$, the \emph{primitivity rank} of
$H$ is 
\[
\pi\left(H\right)=\min\left\{ \rk\left(J\right)\,\middle|\,\begin{gathered}H\le J\le\F_{k}~s.t.\\
H\textrm{ is \textbf{not} a free factor of \ensuremath{J}}
\end{gathered}
\right\} .
\]
Again, if no such $J$ exists, then $\pi\left(H\right)=\infty$. We
call a subgroup $J$ for which the minimum is obtained \textbf{$H$-critical},
and denote the set of $H$-critical subgroups by $\crit\left(H\right)$.
The set of $w$-critical subgroups of a word $w$ is defined analogously.
\end{defn}
Note that for $w\ne1$, $\pi\left(w\right)=\pi\left(\left\langle w\right\rangle \right)$.
Let us give a few examples: $\pi\left(w\right)=0$ iff $w=1$; $\pi\left(w\right)=\infty$
iff $w$ is primitive, and $\pi\left(H\right)=\infty$ iff $H$ is
a free factor; $\pi\left(w\right)=1$ if and only if $w$ is a proper
power, namely $w=v^{d}$ for some $v\in\mathbf{F}_{k}$ and $d\ge2$,
and then $\crit\left(w\right)=\left\{ \left\langle v^{m}\right\rangle \,:\, m\mid d,\,1\leq m<d\right\} $
(assuming that $v$ itself is not a power). By \cite[Lemma 6.8]{Pud14a},
$\pi\left(x_{1}^{\;2}\ldots x_{r}^{\;2}\right)=r$ for every $1\leq r\leq k$.
We thus have that $\pi$ takes all values in $\left\{ 0,1,2,\ldots,k\right\} \cup\left\{ \infty\right\} $,
and Claim \claref{ff-properties}\enuref{intermideate-in-free} shows
that these are all the values it obtains. The primitivity rank of
a word or a subgroup is computable - this is shown in Section \secref{Algebraic-Extensions}.
The distribution of the primitivity rank is discussed in \cite{Pud14b}.

In this paper we sometimes find it more convenient to deal with \emph{reduced
ranks} of subgroups: $\rrk\left(H\right)\overset{{\scriptscriptstyle def}}{=}\rk\left(H\right)-1$.
We therefore define analogously the \emph{reduced primitivity rank},\emph{
}$\rp\left(\cdot\right)\overset{{\scriptscriptstyle def}}{=}\pi\left(\cdot\right)-1$.\medskip{}

As mentioned above, our main result follows from an analysis of the
average number of common fixed points of $\alpha_{n}\left(H\right)$
(where $\alpha_{n}$ denotes a uniformly distributed random homomorphism
in $\Hom\left(\F_{k},S_{n}\right)$). In other words, we count the
number of elements in $\left\{ 1,\ldots,n\right\} $ stabilized by
the images under $\alpha_{n}$ of \emph{all} elements of $H$. Theorem
\nameref{thm:mp_is_prim_ext} follows from the main result of this
analysis:
\begin{thm}
\label{thm:avg_fixed_points}The average number of common fixed points
of $\alpha_{n}\left(H\right)$ is 
\[
\frac{1}{n^{\rrk\left(H\right)}}+\frac{\left|\crit\left(H\right)\right|}{n^{\rp\left(H\right)}}+O\left(\frac{1}{n^{\rp\left(H\right)+1}}\right).
\]
In particular, for a word $w$
\[
\mathbb{E}\left[\#\mathrm{fix}\left(\alpha_{n}\left(w\right)\right)\right]=1+\frac{\left|\crit\left(w\right)\right|}{n^{\rp\left(w\right)}}+O\left(\frac{1}{n^{\rp\left(w\right)+1}}\right).
\]

\end{thm}
We remark that $\crit\left(H\right)$ is always finite (see Section
\ref{sec:Algebraic-Extensions}). Table \tabref{Primitivity-Rank-and-fixed-points}
summarizes the connection implied by Theorem \thmref{avg_fixed_points}
between the primitivity rank of $w$ and the average number of fixed
points in the random permutation $\alpha_{n}\left(w\right)$.

\begin{table}[h]
\begin{centering}
\begin{tabular}{|c|c|c|}
\hline 
$\pi\left(w\right)$ & Description of $w$ & $\mathbb{E}\left[\#\mathrm{fix}\left(\alpha_{n}\left(w\right)\right)\right]$\tabularnewline
\hline 
\hline 
$0$ & $w=1$ & $n$\tabularnewline
\hline 
$1\vphantom{\Big[}$ & $w$ is a power & $1+|\crit\left(w\right)|+O\left(\frac{1}{n}\right)$\tabularnewline
\hline 
$2\vphantom{\Big[}$ & E.g.\ $\left[x_{1},x_{2}\right],x_{1}^{\,2}x_{2}^{\,2}$ & $1+\frac{|\crit\left(w\right)|}{n}+O\left(\frac{1}{n^{2}}\right)$\tabularnewline
\hline 
$3\vphantom{\Big[}$ &  & $1+\frac{|\crit\left(w\right)|}{n^{2}}+O\left(\frac{1}{n^{3}}\right)$\tabularnewline
\hline 
$\vdots$ &  & $\vdots$\tabularnewline
\hline 
$k\vphantom{\Big[}$ & E.g.\ $x_{1}^{\,2}\ldots x_{k}^{\,2}$ & $1+\frac{|\crit\left(w\right)|}{n^{k-1}}+O\left(\frac{1}{n^{k}}\right)$\tabularnewline
\hline 
$\infty$ & $w$ is primitive & $1$\tabularnewline
\hline 
\end{tabular}
\par\end{centering}

\centering{}\caption{Primitivity Rank and Average Number of Fixed Points.\label{tab:Primitivity-Rank-and-fixed-points}}
\end{table}

Theorem \thmref{avg_fixed_points} implies the following general corollary
regarding the family of distributions of $S_{n}$ induced by word
maps:
\begin{cor}
\label{cor:fixed-points-at-least-1}For a non-primitive $w\in\F_{k}$
the average number of fixed points in $\alpha_{n}\left(w\right)$
is strictly greater than 1, for large enough $n$.
\end{cor}
Corollary \corref{fixed-points-at-least-1} is in fact the missing
piece $\left(5\right)\Rightarrow\left(1\right)$ in Theorem \nameref{thm:mp_is_prim_ext}.
In addition, it follows from this corollary that for every $w\in\F_{k}$
and large enough $n$, the average number of fixed points in $\alpha_{n}\left(w\right)$
is at least one%
\footnote{It is suggestive to ask whether this holds for \emph{all }$n$. Namely,
is it true that for every $w\in\F_{k}$ and every $n$, the average
number of fixed points in $\alpha_{n}\left(w\right)$ is at least
1? By results of Abért (\cite{Abe06}), this statement turns out to
be false.%
}. In other words, primitive words generically induce a distribution
of $S_{n}$ with the fewest fixed points on average. 

The results stated above validate completely the conjectural picture
described in \cite{Pud14a}. Theorem \thmref{avg_fixed_points} and
its consequences, Corollaries \corref{prim_in_F_k_iff_prim_in_hF_k},
\corref{prim_are_closed} and \corref{fixed-points-at-least-1}, are
stated there as conjectures (Conjectures 1.10, 7.1, 7.2 and 8.2).\medskip{}

The analysis of the average number of fixed points in $\alpha_{n}\left(w\right)$
has its roots in \cite{Nic94}. Nica notices that by studying the
various quotients of a labeled cycle-graph (corresponding to $w$),
one can compute a rational expression which gives this average for
every large enough $n$. When $w=v^{d}$ with $d$ maximal (so $v$
is not a power), he shows that the limit distribution of the number
of fixed points in $\alpha_{n}\left(w\right)$ (as $n\to\infty$)
is $\delta\left(d\right)+O\left(\frac{1}{n}\right)$, where $\delta\left(d\right)$
is the number of divisors of $d$ (\cite{Nic94}, Corollary 1.3)%
\footnote{Nica's result is in fact more general: the same statement holds not
only for fixed points but for cycles of length $L$ for every fixed
$L$.%
}. Nica's result follows from Theorem \thmref{avg_fixed_points}: if
$w\neq1$ is a proper power and $w=v^{d}$ with $d\geq2$ maximal,
then $|\crit\left(w\right)|=\delta\left(d\right)-1$, and if it is
not a power then $\rp\left(w\right)\geq1$. \medskip{}

The results of this paper have interesting implications in the study
of expansion in random graphs: In \cite{Pud14b}, the first author
presents a new approach to showing that random graphs are nearly optimal
expanders. A crucial ingredient in the proof is Theorem \thmref{avg_fixed_points}.
More particularly, it was conjectured by Alon \cite{Alo86} that the
spectral gap of a random $d$-regular graph is a.a.s.\ arbitrarily
close to $d-2\sqrt{d-1}$, and this conjecture was generalized by
Friedman \cite{Fri03} to non-regular graphs. In \cite{Fri08}, Alon's
conjecture is proved by highly sophisticated arguments, which are
not applicable for the generalized conjecture (as far as is known).
The results in \cite{Pud14b} give a simple proof which nearly recovers
Friedman's results regarding Alon's conjecture, and can be applied
also for the generalized conjecture, giving the best results as of
now regarding non-regular graphs.

\section{Overview of the proof\label{sec:Overview-of-the}}

The proof of our main theorem involves several structures of posets
(partially ordered sets) on $\mathfrak{sub}_{f\! g}\left(\mathbf{F}_{k}\right)$,
the set of finitely generated subgroups of $\F_{k}$. This set has,
of course, a natural structure of a poset given by the relation of
inclusion. However, there are other interesting partial orders defined
on it: the relation of \emph{algebraic extensions}, and the family
of relations defined by \emph{covers}. We introduce some notation:
If $\preceq$ is some partial order on $\mathfrak{sub}_{f\! g}\left(\mathbf{F}_{k}\right)$,
and $H,J\fg\F_{k}$, we define the \emph{closed interval} 
\[
\left[H,J\right]_{\preceq}=\left\{ L\in\mathfrak{sub}_{f\! g}\left(\mathbf{F}_{k}\right)\,\middle|\, H\preceq L\preceq J\right\} 
\]
and similarly the open interval $\left(H,J\right)_{\preceq}=\left\{ L\,\middle|\, H\precneqq L\precneqq J\right\} $,
the half-bounded interval $\left[H,\infty\right)_{\preceq}=\left\{ L\,\middle|\, H\preceq L\right\} $,
and so on (see also the glossary).

\paragraph*{Algebraic\ Extensions}

This notion goes back to \cite{takahasi1951note}, and was further
studied in \cite{KM02,MVW07}.
\begin{defn}
\label{def:algebraic-extension}We say that $J$ is an \emph{algebraic
extension }of $H$, denoted $H\alg J$, if $H\le J$ and $H$ is not
contained in any proper free factor of $J$.
\end{defn}
The terminology comes from similarities (that go only to some extent)
between this notion and that of algebraic extensions of fields (in
this line of thought, $J$ is a \emph{transcendental} \emph{extension
}of\emph{ $H$ }when $H\ff J$). We devote Section \secref{Algebraic-Extensions}
to study this relation. It is clearly reflexive and antisymmetric,
but it is also transitive (Claim \claref{alg-transitive}). In addition,
it is very sparse: it turns out that $\AE{H}$, the set of algebraic
extensions of $H$, is finite for every $H\fg\F_{k}$, so in particular
$\left(\mathfrak{sub}_{f\! g}\left(\mathbf{F}_{k}\right),\alg\right)$
is locally finite%
\footnote{A \emph{locally finite} poset is one in which every closed interval
$[a,b]=\{x\,:\, a\le x\le b\}$ is finite.%
}. It is a simple observation that $H$-critical subgroups are in particular
algebraic extensions of $H$, i.e.\ $\crit\left(H\right)\subseteq\AE{H}$.
In fact, they are the proper algebraic extensions of minimal rank.

\paragraph*{$X$-cover}

For every basis $X=\left\{ x_{1},\ldots,x_{k}\right\} $ of $\F_{k}$
there is a partial order denoted $\covers$, which is based on the
notion of quotients, or surjective morphisms, of \emph{core graphs}.
Introduced in \cite{Sta83}, core graphs provide a geometric approach
to the study of free groups (for an extensive survey see \cite{KM02},
and also \cite{MVW07} and the references therein). Given the basis
$X$, Stallings associates with every $H\le\F_{k}$ a directed and
pointed graph denoted $\G_{X}\left(H\right)$, whose edges are labeled
by the elements of $X$. A full definition appears in Section \secref{Core-Graphs},
but we illustrate the concept in Figure \figref{first_core_graph}.
It shows the core graph of the subgroup of $\F_{2}$ generated by
$x_{1}x_{2}^{-1}x_{1}$ and $x_{1}^{-2}x_{2}$, with $X=\left\{ x_{1},x_{2}\right\} $.\FigBesBeg \\
\begin{figure}[h]
\centering{}%
\begin{minipage}[t]{0.4\columnwidth}%
\[
\xymatrix{\otimes\ar[rr]^{{\textstyle x_{1}}} &  & \bullet\\
\\
\bullet\ar[rr]^{{\textstyle x_{1}}}\ar[uu]_{{\textstyle x_{2}}} &  & \bullet\ar[uull]_{{\textstyle x_{1}}}\ar[uu]_{{\textstyle x_{2}}}
}
\]
\end{minipage}\caption{\label{fig:first_core_graph} The core graph $\G_{X}\left(H\right)$
where $X=\left\{ x_{1},x_{2}\right\} $ and $H=\left\langle x_{1}x_{2}^{-1}x_{1},x_{1}^{-2}x_{2}\right\rangle \leq\F_{2}$.}
\end{figure}
\FigBesEnd \\
The order $\covers$ is defined as follows: for $H,J\leq\mathbf{F}_{k}$
one has $H\covers J$ iff the associated core graph $\G_{X}\left(J\right)$
is a quotient (as a pointed labeled graph) of the core graph $\G_{X}\left(H\right)$
(see Definition \defref{quotient}). When $H\fg\F_{k}$, $\G_{X}\left(H\right)$
is finite (Claim \claref{core-graphs-properties}\enuref{finite-rk-graph}),
and thus has only finitely many quotients. As it turns out that different
groups correspond to different core graphs, this implies that $\left(\mathfrak{sub}_{f\! g}\left(\mathbf{F}_{k}\right),\covers\right)$
is locally finite too. We stress that we have here an infinite family
of partial orders, one for every choice of basis for $\F_{k}$. Although
the dependency on the basis makes these orders somewhat less universal,
they turn out to be the most useful for our purposes.

\medskip{}

The various relations between subgroups of $\mathbf{F}_{k}$ are the
following: 
\[
J\in\crit\left(H\right)\:\Rightarrow\: H\alg J\:\Rightarrow\: H\covers J\:\Rightarrow\: H\leq J
\]
for any $H,J\leq\mathbf{F}_{k}$ and any basis $X$ (see Sections
\secref{Core-Graphs} and \secref{Algebraic-Extensions}).\medskip{}

Recall that the main theorems of this paper follow from Theorem \ref{thm:avg_fixed_points},
which estimates the expected number of common fixed points of $\alpha_{n}\left(H\right)$,
where $H\fg\mathbf{F}_{k}$ and $\alpha_{n}$ is a random homomorphism
in $\Hom\left(\mathbf{F}_{k},S_{n}\right)$. This result is achieved
by studying a broader question: For every pair of $H,J\fg\F_{k}$
such that $H\le J$, we define for $n\in\mathbb{N}$ 
\begin{equation}
\Phi_{H,J}\left(n\right)=\textrm{The expected number of common fixed points of \ensuremath{\alpha_{J,n}\left(H\right)}},\label{eq:Phi}
\end{equation}
where $\alpha_{J,n}\in\Hom\left(J,S_{n}\right)$ is a random homomorphism
(chosen with uniform distribution). In this perspective, Nica finds
$\lim_{n\rightarrow\infty}\Phi_{\left\langle w\right\rangle ,\F_{k}}\left(n\right)$,
and shows that it separates powers and non-powers. Theorem \ref{thm:avg_fixed_points}
shows that the first two terms in the expansion of $\Phi_{\left\langle w\right\rangle ,\mathbf{F}_{k}}\left(n\right)$
yield $w$'s primitivity rank, which in particular distinguishes powers
($\pi\left(w\right)=1$) and primitives ($\pi\left(w\right)=\infty$).
Furthermore, the same holds for subgroups using $\Phi_{H,\mathbf{F}_{k}}\left(n\right)$.

\begin{wrapfigure}{o}{0.3\columnwidth}%
\vspace{-6mm}
\[
\xymatrix{ & \Phi\ar@{-}[dl]\ar@{-}[dr]\\
L^{X}\ar@{-}[dr] &  & R^{X}\ar@{-}[dl]\\
 & C^{X}
}
\]
\end{wrapfigure}%
As remarked, in order to understand $\Phi_{H,\mathbf{F}_{k}}$ we
turn to analyze the totality of functions $\Phi_{H,J}$, for various
$H\leq J\leq\mathbf{F}_{k}$. We apply the machinery of Möbius inversions
to the incidence algebra arising from the locally finite poset $\left(\mathfrak{sub}_{f\! g}\left(\mathbf{F}_{k}\right),\covers\right)$.
The local finiteness of the order $\covers$ allows us to ``derive''
the function $\Phi$ and obtain its ``right derivation'' $R^{X}$,
its ``left derivation'' $L^{X}$, and its ``two sided derivation''
$C^{X}$ (see Section \ref{sec:Mobius-Inversions}). For instance,
$\Phi_{H,J}$ can be presented as finite sums of $R^{X}$:

\[
\Phi_{H,J}=\sum\limits _{M\in\XC{H}{J}}R_{H,M}^{X}
\]
(here $\XC{H}{J}$ is an abbreviation for $\left[H,J\right]_{\leq_{\Xcov}}$,
i.e.\ $\XC{H}{J}=\left\{ M\,\middle|\, H\covers M\covers J\right\} $). 

The proof of Theorem \thmref{avg_fixed_points} is then based on a
series of lemmas and propositions characterizing $\Phi$ and its three
derivations:
\begin{itemize}
\item (Proposition \propref{R-independent-of-basis}) The right derivation
$R^{X}$ is supported on algebraic extensions, i.e.\ if $H\covers M$
but $M$ is not an algebraic extension of $H$ then $R_{H,M}^{X}\equiv0$. 
\item (The discussion in Section \ref{sec:Random-Covers}) The random homomorphism
$\alpha_{J,n}\in\mathrm{Hom}\left(J,S_{n}\right)$ can be encoded
as a random covering-space $\widehat{\Gamma}$ of the core graph $\Gamma_{X}\left(J\right)$,
and $\Phi_{H,J}\left(n\right)$ can then be interpreted as the expected
number of lifts of $\Gamma_{X}\left(H\right)$ into $\widehat{\Gamma}$.
\item (Lemmas \ref{lem:L-as-lifts} and \ref{lem:L-formula}) The left derivation
$L^{X}$ is the expected number of \emph{injective} lifts of the core
graph $\Gamma_{X}\left(H\right)$ into the random covering $\widehat{\Gamma}$
of the core graph $\Gamma_{X}\left(J\right)$, and a rational expression
can be computed for $L_{H,J}^{X}$.
\item (Proposition \ref{prop:order-magnitude-C} and Section \ref{sub:C})
An analysis involving Stirling numbers of the rational expressions
for $L^{X}$ yields a combinatorial meaning for the two-sided derivation
$C^{X}$. Using the classification of primitivity rank we then obtain
a first-order estimate for the size of $C_{H,J}^{X}$.
\item (Proposition \propref{order-magnitude-R}) From $C^{X}$ we return
to $R^{X}$ (by ``left-integration''), obtaining that whenever $H\alg M$
we have 
\[
R_{H,M}^{X}=\frac{1}{n^{\rrk(M)}}+O\left(\frac{1}{n^{\rrk(M)+1}}\right)
\]
and by right integration of $R^{X}$, we obtain the order of magnitude
of $\Phi$, which was our goal.
\end{itemize}
The paper is arranged as follows: in Section \ref{sec:Core-Graphs}
the notion of core graphs is explained in details, as well as the
partial order $\covers$ and some of the results from \cite{Pud14a}
which are used here. In Section \ref{sec:Algebraic-Extensions} we
survey the main properties of algebraic extensions of free groups.
Section \ref{sec:Mobius-Inversions} is devoted to recalling Möbius
derivations on locally-finite posets and introducing the different
derivations of $\Phi$. In Section \ref{sec:Random-Covers} we discuss
the connection of the problem to random coverings of graphs and analyze
the left derivation $L^{X}$. The proof of Theorem \ref{thm:avg_fixed_points}
is completed in Section \ref{sec:proof} via the analysis of the two-sided
derivation $C^{X}$ and the consequence of the latter on the right
derivation $R^{X}$. Finally, corollaries of our results to the field
of profinite groups, and to decidability questions in group theory,
are discussed in Section \ref{sec:Profinite}. We finish with a list
of open problems naturally arising from this paper. For the reader's
convenience, there is also a glossary of notions and notations at
the end of this manuscript.

\section{\label{sec:Core-Graphs}Core graphs and the partial order of covers}

Fix a basis $X=\left\{ x_{1},\ldots,x_{k}\right\} $ of $\F_{k}$.
Associated with every subgroup $H\le\F_{k}$ is a directed, pointed
graph whose edges are labeled by $X$. This graph is called \emph{the
(Stallings) core-graph associated with $H$} and is denoted by $\G_{X}\left(H\right)$.
We recall the notion of the Schreier (right) coset graph of $H$ with
respect to the basis $X$, denoted by $\overline{\G}_{X}\left(H\right)$.
This is a directed, pointed and edge-labeled graph. Its vertex set
is the set of all right cosets of $H$ in $\F_{k}$, where the basepoint
corresponds to the trivial coset $H$. For every coset $Hw$ and every
basis-element $x_{j}$ there is a directed $j$-edge (short for $x_{j}$-edge)
going from the vertex $Hw$ to the vertex $Hwx_{j}$.%
\footnote{Alternatively, $\overline{\G}_{X}\left(H\right)$ is the quotient
$H\backslash T$, where $T$ is the Cayley graph of $\F_{k}$ with
respect to the basis $X$, and $F_{k}$ (and thus also H) acts on
this graph from the left. Moreover, this is the covering-space of
$\overline{\G}_{X}\left(F_{k}\right)=\Gamma_{X}\left(F_{k}\right)$,
the bouquet of k loops, corresponding to $H$, via the correspondence
between pointed covering spaces of a space $Y$ and subgroups of its
fundamental group $\pi_{1}\left(Y\right)$.%
}

The core graph $\G_{X}\left(H\right)$ is obtained from $\overline{\G}_{X}\left(H\right)$
by omitting all the vertices and edges of $\overline{\G}_{X}\left(H\right)$
which are not traced by any reduced (i.e., non-backtracking) path
that starts and ends at the basepoint. Stated informally, we trim
all ``hanging trees'' from $\overline{\G}_{X}\left(H\right)$. Formally,
$\G_{X}\left(H\right)$ is the induced subgraph of $\overline{\G}_{X}\left(H\right)$
whose vertices are all cosets $Hw$ (with $w$ reduced), such that
for some word $w'$ the concatenation $ww'$ is reduced, and $w\cdot w'\in H$.
To illustrate, Figure \figref{coset_and_core_graphs} shows the graphs
$\overline{\G}_{X}\left(H\right)$ and $\G_{X}\left(H\right)$ for
$H=\langle x_{1}x_{2}x_{1}^{-3},x_{1}^{\;2}x_{2}x_{1}^{-2}\rangle\leq\F_{2}$.
Note that the graph $\overline{\G}_{X}\left(H\right)$ is $2k$-regular:
every vertex has exactly one outgoing $j$-edge and one incoming $j$-edge,
for every $1\le j\le k$. Every vertex of $\G_{X}\left(H\right)$
has \emph{at most} one outgoing $j$-edge and\emph{ at most} one incoming
$j$-edge, for every $1\le j\le k$.

\begin{figure}[h]
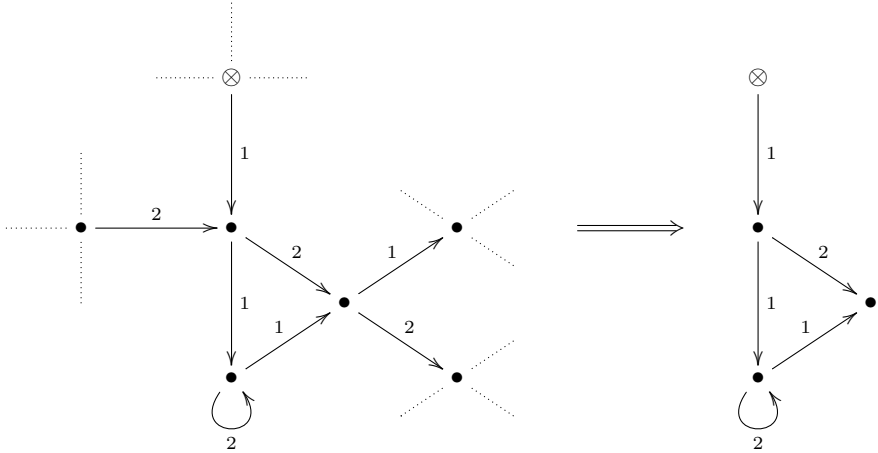

\begin{centering}
\begin{center}
\xy
(30,60)*+{\otimes}="s0";
(30,40)*+{\bullet}="s1";%
(45,30)*+{\bullet}="s2";%
(30,20)*+{\bullet}="s3";%
{\ar^{1} "s0";"s1"};%
{\ar^{2} "s1";"s2"};%
{\ar^{1} "s3";"s2"};%
{\ar^{1} "s1";"s3"};%
{\ar@(dl,dr)_{2} "s3";"s3"};%
(10,40)*+{\bullet}="t0";%
(60,40)*+{\bullet}="t1";%
(60,20)*+{\bullet}="t2";%
{\ar^{2} "t0";"s1"};%
{\ar^{1} "s2";"t1"};%
{\ar^{2} "s2";"t2"};%
{\ar@{..} "s0"; (20,60)*{}};%
{\ar@{..} "s0"; (30,70)*{}};%
{\ar@{..} "s0"; (40,60)*{}};%
{\ar@{..} "t0"; (10,50)*{}};%
{\ar@{..} "t0"; (0,40)*{}};%
{\ar@{..} "t0"; (10,30)*{}};%
{\ar@{..} "t1"; (52.5,45)*{}};%
{\ar@{..} "t1"; (67.5,45)*{}};%
{\ar@{..} "t1"; (67.5,35)*{}};%
{\ar@{..} "t2"; (67.5,25)*{}};%
{\ar@{..} "t2"; (67.5,15)*{}};%
{\ar@{..} "t2"; (52.5,15)*{}};%
{\ar@{=>} (76,40)*{}; (90,40)*{}};
(100,60)*+{\otimes}="m0";%
(100,40)*+{\bullet}="m1";%
(115,30)*+{\bullet}="m2";%
(100,20)*+{\bullet}="m3";%
{\ar^{1} "m0";"m1"};%
{\ar^{2} "m1";"m2"};%
{\ar^{1} "m3";"m2"};%
{\ar^{1} "m1";"m3"};%
{\ar@(dl,dr)_{2} "m3";"m3"};%
\endxy
\par\end{center}
\par\end{centering}

\caption{$\overline{\G}_{X}\left(H\right)$ and $\G_{X}\left(H\right)$ for
$H=\langle x_{1}x_{2}x_{1}^{-3},x_{1}^{\;2}x_{2}x_{1}^{-2}\rangle\leq\F_{2}$.
The Schreier coset graph $\overline{\G}_{X}\left(H\right)$ is the
infinite graph on the left (the dotted lines represent infinite $4$-regular
trees). The basepoint ``$\otimes$'' corresponds to the trivial
coset $H$, the vertex below it corresponds to the coset $Hx_{1}$,
the one further down corresponds to $Hx_{1}^{\;2}=Hx_{1}x_{2}x_{1}^{-1}$,
etc. The core graph $\G_{X}\left(H\right)$ is the finite graph on
the right, which is obtained from $\overline{\G}_{X}\left(H\right)$
by omitting all vertices and edges that are not traced by reduced
closed paths around the basepoint.}

\label{fig:coset_and_core_graphs} 
\end{figure}

If $\Gamma$ is a directed pointed graph labeled by some set $X$,
paths in $\Gamma$ correspond to words in $\mathbf{F}\left(X\right)$
(the free group generated by $X$). For instance, the path (from left
to right)
\[
\xymatrix@C=35pt{\bullet\ar[r]^{x_{2}} & \bullet\ar[r]^{x_{2}} & \bullet\ar[r]^{x_{1}} & \bullet & \bullet\ar[l]_{x_{2}}\ar[r]^{x_{3}} & \bullet & \bullet\ar[l]_{x_{1}}}
\]
corresponds to the word $x_{2}^{\;2}x_{1}x_{2}^{-1}x_{3}x_{1}^{-1}$.
The set of all words obtained from closed paths around the basepoint
in $\Gamma$ is a subgroup of $\F\left(X\right)$ which we call the
\emph{labeled fundamental group }of $\Gamma$, and denote by $\pi_{1}^{X}\left(\Gamma\right)$.
Note that $\pi_{1}^{X}\left(\Gamma\right)$ need not be isomorphic
to $\pi_{1}\left(\Gamma\right)$, the standard fundamental group of
$\Gamma$ viewed as a topological space: for example, take $\vphantom{\Big|}\Gamma=\xymatrix@1{\otimes\ar@(dl,ul)[]^{x_{1}}\ar@(dr,ur)[]_{x_{1}}}$.

However, it is not hard to show that when $\Gamma$ is a core graph,
then $\pi_{1}^{X}\left(\Gamma\right)$ \emph{is }isomorphic to $\pi_{1}\left(\Gamma\right)$
(e.g.\ \cite{MVW07}). In this case the labeling gives a canonical
identification of $\pi_{1}\left(\Gamma\right)$ as a subgroup of $\mathbf{F}\left(X\right)$.
It is an easy observation that 
\begin{equation}
\pi_{1}^{X}\left(\overline{\G}_{X}\left(H\right)\right)=\pi_{1}^{X}\left(\G_{X}\left(H\right)\right)=H\label{eq:canon_iso}
\end{equation}
This gives a one-to-one correspondence between subgroups of $\mathbf{F}\left(X\right)=\F_{k}$
and core graphs labeled by $X$. Namely, $\pi_{1}^{X}$ and $\Gamma_{X}$
are the inverses of each other in a bijection (Galois correspondence)
\begin{equation}
\left\{ {\mathrm{Subgroups}\atop \mathrm{of}\,\mathbf{F}\left(X\right)}\right\} \:{\underrightarrow{\;\Gamma_{X}\;}\atop \overleftarrow{\;\pi_{1}^{X}\;}}\:\left\{ {\mathrm{Core\, graphs}\atop \mathrm{labeled\, by}\, X}\right\} \label{eq:pi_1_gamma}
\end{equation}
Core graphs were introduced by Stallings \cite{Sta83}. Our definition
is slightly different, and closer to the one in \cite{KM02,MVW07}
in that we allow the basepoint to be of degree one, and in that our
graphs are directed and edge-labeled. We remark that it is possible
to study core graphs from a purely combinatorial point of view, as
labeled pointed connected graphs satisfying 
\begin{enumerate}
\item No two equally labeled edges originate or terminate at the same vertex.
\item Every vertex and edge are traced by some non-backtracking closed path
around the basepoint.
\end{enumerate}
Starting with this definition, every choice of an ordered basis for
$\F_{k}$ then gives a correspondence between these graphs and subgroups
of $\F_{k}$.

\medskip{}

In this paper we are mainly interested in finite core graphs, and
we now list some basic properties of these (proofs can be found in
\cite{Sta83,KM02,MVW07}).
\begin{claim}
\label{cla:core-graphs-properties} Let $H$ be a subgroup of $\F_{k}$
with an associated core graph $\G=\G_{X}\left(H\right)$. The Euler
Characteristic of a graph, denoted $\chi\left(\cdot\right)$, is the
number of vertices minus the number of edges.
\begin{enumerate}
\item \label{enu:finite-rk-graph}$\rk\left(H\right)<\infty\Longleftrightarrow\G$
is finite.
\item \label{enu:euler}$\rrk\left(H\right)=-\chi\left(\G\right)$.
\item The correspondence \eqref{pi_1_gamma} restricts to a correspondence
between $\mathfrak{sub}_{f\! g}\left(\mathbf{F}_{k}\right)$ and finite
core graphs. 
\end{enumerate}
\end{claim}
Given a finite set of words $\left\{ h_{1},\ldots,h_{m}\right\} \subseteq\F\left(X\right)$
that generate a subgroup $H$, the core graph $\Gamma_{X}\left(H\right)$
can be algorithmically constructed as follows. Every $h_{i}$ corresponds
to some path with directed edges labeled by the $x_{j}$'s (we assume
the elements are given in reduced forms, otherwise we might need to
prune leaves at the end of the algorithm). Merge these $m$ paths
to a single graph (bouquet) by identifying all their $2m$ end-points
to a single vertex, which is marked as the basepoint. The labeled
fundamental group of this graph is clearly $H$. Then, as long as
there are two $j$-labeled edges with the same terminus (resp.\ origin)
for some $j$, merge the two edges and their origins (resp.\ termini).
Such a step is often referred to as \emph{Stallings folding}. It is
fairly easy to see that each folding step does not change the labeled
fundamental group of the graph, that the resulting graph is indeed
$\G_{X}\left(H\right)$, and that the order of folding has no significance.
To illustrate, we draw in Figure \figref{folding_process} a folding
process by which we obtain the core graph $\G_{X}\left(H\right)$
of $H=\langle x_{1}x_{2}x_{1}^{-3},x_{1}^{\;2}x_{2}x_{1}^{-2}\rangle\leq\F_{2}$
from the given generating set.

\begin{figure}[h]
\begin{centering}
\includegraphics[width=0.99\columnwidth]{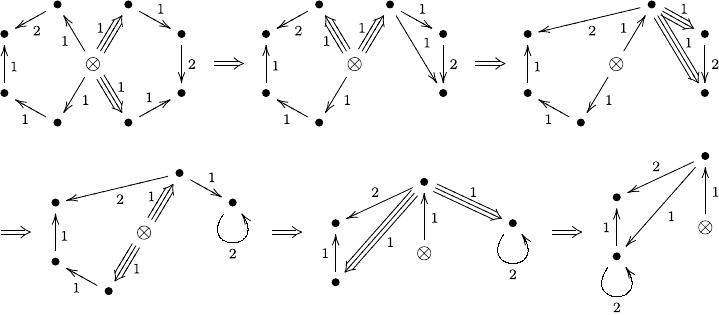}
\par\end{centering}

\caption{\label{fig:folding_process} Constructing the core graph $\G_{X}\left(H\right)$
of $H=\langle x_{1}x_{2}x_{1}^{-3},x_{1}^{\;2}x_{2}x_{1}^{-2}\rangle\leq\F_{2}$
from the given generating set. We start with the upper left graph
which contains a distinct loop at the basepoint for each (reduced)
element of the generating set. Then, at an arbitrary order, we merge
pairs of equally-labeled edges which share the same origin or the
same terminus (here we mark by triple arrows the pair of edges being
merged next). The graph at the bottom right is $\G_{X}\left(H\right)$,
as it has no equally-labeled edges sharing the same origin or terminus.}
\end{figure}

A \emph{morphism} between two core-graphs is a map that sends vertices
to vertices and edges to edges, and preserves the structure of the
core graphs. Namely, it preserves the incidence relations, sends the
basepoint to the basepoint, and preserves the directions and labels
of the edges. 

As in Claim \claref{core-graphs-properties}, each of the following
properties is either proven in (some of) \cite{Sta83,KM02,MVW07}
or an easy observation:
\begin{claim}
\label{cla:morphism-properties} Let $H,J,L\le\F_{k}$ be subgroups.
Then 
\begin{enumerate}
\item A morphism $\G_{X}\left(H\right)\to\G_{X}\left(J\right)$ exists if
and only if $H\leq J$. 
\item If a morphism $\G_{X}\left(H\right)\to\G_{X}\left(J\right)$ exists,
it is unique. We denote it by $\eta_{H\to J}^{X}$.
\item Whenever $H\le L\le J$, $\eta_{H\to J}^{X}=\eta_{L\to J}^{X}\circ\eta_{H\to L}^{X}$.%
\footnote{Points (1)-(3) can be formulated by saying that \eqref{pi_1_gamma}
is in fact an isomorphism of categories, given by the functors $\pi_{1}^{X}$
and $\Gamma_{X}$.%
}
\item If $\eta_{H\to J}^{X}$ is injective, then $H\ff J$.%
\footnote{But not vice-versa: for example, consider $\left\langle x_{1}x_{2}^{\,2}\right\rangle \ff\F_{2}$.%
}
\item Every morphism is an immersion (locally injective at the vertices). 
\end{enumerate}
\end{claim}
A special role is played by \emph{surjective} morphisms of core graphs:
\begin{defn}
\label{def:quotient}Let $H\le J\le\F_{k}$. Whenever $\eta_{H\to J}^{X}$
is surjective, we say that \emph{$\G_{X}\left(H\right)$ covers $\G_{X}\left(J\right)$}
or that \emph{$\G_{X}\left(J\right)$ is a quotient of $\G_{X}\left(H\right)$}.
We indicate this by $\G_{X}\left(H\right)\th\G_{X}\left(J\right)$.
As for the groups, we say that \emph{$H$ $X$-covers $J$} and denote
this by $H\covers J$\emph{.}
\end{defn}
By ``surjective'' we mean surjective on both vertices and edges.
Note that we use the term ``covers'' even though in general this
is \emph{not} a topological covering map (a morphism between core
graphs is always locally injective at the vertices, but it need not
be locally bijective). In Section \secref{Random-Covers} we do study
topological covering maps, and we reserve the term ``coverings''
for these.

For instance, $H=\langle x_{1}x_{2}x_{1}^{-3},x_{1}^{\;2}x_{2}x_{1}^{-2}\rangle\le\F_{k}$
$X$-covers the group $J=\langle x_{2},x_{1}^{\;2},x_{1}x_{2}x_{1}\rangle$,
the corresponding core graphs of which are the leftmost and rightmost
graphs in Figure \figref{quotient-graph}. As another example, a core
graph $\G$ $X$-covers $\G_{X}\left(\F_{k}\right)$ (which is merely
a wedge of $k$ loops) if and only if it contains edges of all $k$
labels.

As implied by the notation, the relation $H\covers J$ indeed depends
on the given basis\emph{ $X$ }of\emph{ $\F_{k}$}. For example, if
$H=\langle x_{1}x_{2}\rangle$ then $H\covers\F_{2}$. However, for
$Y=\left\{ x_{1}x_{2},x_{2}\right\} $, $H$ does not\emph{ }$Y$-cover
$\F_{2}$, as $\G_{Y}\left(H\right)$ consists of a single vertex
and a single loop and has no quotients apart from itself.

It is easy to see that the relation ``$\covers$'' indeed constitutes
a partial ordering of the set of subgroups of $\F_{k}$. We make a
few other useful observations:
\begin{claim}
\label{cla:cover-properties}Let $H,J,L\le\F_{k}$ be subgroups. Then 
\begin{enumerate}
\item Whenever $H\le J$ there exists an intermediate subgroup $M$ such
that $H\covers M\ff J$. 
\item If one adds the condition that $\Gamma_{X}\left(M\right)$ embeds
in $\Gamma_{X}\left(J\right)$, then this $M$ is unique.
\item \label{enu:Cover-insideCover}If $H\covers J$ and $H\covers L\le J$,
then $L\covers J$.
\item \label{enu:O_X(H)-is-finite}If $H$ is finitely generated then it
$X$-covers only a finite number of groups. In particular, the poset
$\left(\mathfrak{sub}_{f\! g}\left(\mathbf{F}_{k}\right),\covers\right)$
is locally finite.
\end{enumerate}
\end{claim}
\begin{proof}
Point (1) follows from the factorization of the morphism $\eta_{H\to J}^{X}$
to a surjection followed by an embedding. Indeed, it is easy to see
that the image of $\eta_{H\to J}^{X}$ is a sub-graph of $\G_{X}\left(J\right)$
which is in itself a core graph. Namely, it contains no ``hanging
trees'' (edges and vertices not traced by reduced paths around the
basepoint). Let $M=\pi_{1}^{X}\left(\mathrm{im}\,\eta_{H\to J}^{X}\right)$
be the subgroup corresponding to this sub-core-graph. (1) now follows
from points (1) and (4) in Claim \claref{morphism-properties}. Point
(2) follows from the uniqueness of such factorization of a morphism.
Point (3) follows from the fact that if $\eta_{H\to J}^{X}=\eta_{L\to J}^{X}\circ\eta_{H\to L}^{X}$
is surjective then so is $\eta_{L\to J}^{X}$. Point (4) follows from
the fact that $\G_{X}\left(H\right)$ is finite (Claim \claref{core-graphs-properties}\enuref{finite-rk-graph})
and thus has only finitely many quotients, and each quotient correspond
to a single group (by \eqref{pi_1_gamma}).
\end{proof}
In \cite{MVW07}, the set of $X$-quotients of $H$ 
\begin{equation}
\XF{H}=\left\{ J\,\middle|\, H\covers J\right\} \label{eq:def_of_O_H}
\end{equation}
is called the \emph{$X$-fringe} of $H$. Claim \claref{cover-properties}\enuref{O_X(H)-is-finite}
states in this terminology that for every $H\fg\F_{k}$ (and every
basis $X$), $\left|\XF{H}\right|<\infty$. Note that $\XF{H}$ always
contains the supremum of its elements, namely the group generated
by the elements of $X$ which label edges in $\G_{X}\left(H\right)$
(which is $\pi_{1}^{X}\left(\mathrm{im}\,\eta_{H\rightarrow\F_{k}}^{X}\right)$).
(We remark that in the special case of $H=\left\langle w\right\rangle $
for some $w\in\F_{k}$, the set $\XF{\left\langle w\right\rangle }$
appears also in \cite{turner1996test} and, in a very different language,
in the aforementioned \cite{Nic94}.)

It is easy to see that quotients of $\Gamma_{X}\left(H\right)$ are
determined by the partition they induce of the vertex set $V\left(\G_{X}\left(H\right)\right)$.
However, not every partition $P$ of $V\left(\G_{X}\left(H\right)\right)$
corresponds to a quotient core-graph: in the resulting graph, which
we denote by $\nicefrac{\Gamma_{X}\left(H\right)}{P}$, two distinct
$j$-edges may have the same origin or the same terminus. Then again,
when a partition $P$ of $V\left(\G_{X}\left(H\right)\right)$ yields
a quotient which is not a core-graph, we can perform Stallings foldings
(as demonstrated in Figure \figref{folding_process}) until we obtain
a core graph. Since Stallings foldings do not affect $\pi_{1}^{X}$,
the core graph we obtain in this manner is $\Gamma_{X}\left(J\right)$,
where $J=\pi_{1}^{X}\left(\nicefrac{\Gamma_{X}\left(H\right)}{P}\right)$.
The resulting partition $\bar{P}$ of $V\left(\G_{X}\left(H\right)\right)$
(as the fibers of $\eta_{H\rightarrow J}^{X}$) is the finest partition
of $V\left(\G_{X}\left(H\right)\right)$ which gives a quotient core-graph
and which is still coarser than $P$. We illustrate this in Figure
\figref{quotient-graph}.

\begin{figure}[h]
\noindent \begin{centering}
\begin{minipage}[t]{0.9\columnwidth}%
\noindent \begin{center}
\begin{center}
\xy 
(0,35)*+{\otimes}="m0"+(-3,0)*{\scriptstyle v_1};%
(20,35)*+{\bullet}="m1"+(3,0)*{\scriptstyle v_2};%
(20,15)*+{\bullet}="m2"+(3,0)*{\scriptstyle v_3};%
(0,15)*+{\bullet}="m3"+(-3,0)*{\scriptstyle v_4};%
{\ar^{1} "m0";"m1"};%
{\ar^{2} "m1";"m2"};%
{\ar^{1} "m3";"m2"};%
{\ar^{1} "m1";"m3"};%
{\ar@(dl,dr)_{2} "m3";"m3"};%
(40,25)*+{\otimes}="t0"+(-4,4)*{\scriptstyle \{v_1,v_4\}};%
(60,40)*+{\bullet}="t1"+(4,2)*{\scriptstyle \{v_2\}};%
(60,10)*+{\bullet}="t2"+(4,-2)*{\scriptstyle \{v_3\}};%
{\ar@/^1pc/^{1} "t0";"t1"};%
{\ar^{2} "t1";"t2"};%
{\ar^{1} "t0";"t2"};%
{\ar@/^1pc/^{1} "t1";"t0"};%
{\ar@(dl,dr)_{2} "t0";"t0"};%
(80,25)*+{\otimes}="t0"+(-4,4)*{\scriptstyle \{v_1,v_4\}};%
(100,25)*+{\bullet}="t1"+(4,4)*{\scriptstyle \{v_2,v_3\}};%
{\ar@/^1pc/^{1} "t0";"t1"};%
{\ar@/^1pc/^{1} "t1";"t0"};%
{\ar@(dl,dr)_{2} "t0";"t0"};%
{\ar@(dl,dr)_{2} "t1";"t1"};%
\endxy 
\par\end{center}
\par\end{center}%
\end{minipage}
\par\end{centering}

\caption{\label{fig:quotient-graph} The left graph is the core graph $\G_{X}\left(H\right)$
of $H=\left\langle x_{1}x_{2}x_{1}^{-3},x_{1}^{\;2}x_{2}x_{1}^{-2}\right\rangle \leq\F_{2}$.
Its vertices are denoted by $v_{1},\ldots,v_{4}$. The graph in the
middle is the quotient $\nicefrac{\Gamma_{X}\left(H\right)}{P}$ corresponding
to the partition $P=\left\{ \left\{ v_{1},v_{4}\right\} ,\left\{ v_{2}\right\} ,\left\{ v_{3}\right\} \right\} $.
This is not a core graph as there are two $1$-edges originating at
$\left\{ v_{1},v_{4}\right\} $. In order to obtain a core quotient-graph,
we use the Stallings folding process (illustrated in Figure \figref{folding_process}).
The resulting core graph, $\Gamma_{X}\left(\pi_{1}^{X}\left(\nicefrac{\Gamma_{X}\left(H\right)}{P}\right)\right)$,
is shown on the right and corresponds to the partition $\bar{P}=\left\{ \left\{ v_{1},v_{4}\right\} ,\left\{ v_{2},v_{3}\right\} \right\} $.}
\end{figure}

Thus, there is sense in examining the quotient of a core graph $\G$
``generated'' by some partition $P$ of its vertex set, namely,
$\Gamma_{X}\left(\pi_{1}^{X}\left(\nicefrac{\Gamma}{P}\right)\right)$.
The most interesting case is that of the ``simplest'' partitions:
those which identify only a single pair of vertices. Before looking
at these, we introduce a measure for the complexity of partitions:
if $P\subseteq2^{\mathcal{X}}$ is a partition of some set $\mathcal{X}$,
let 
\begin{equation}
\left\Vert P\right\Vert \overset{{\scriptscriptstyle def}}{=}\left|\mathcal{X}\right|-\left|P\right|=\sum_{B\in P}\left(\left|B\right|-1\right).\label{eq:Partition_norm}
\end{equation}
Namely, $\left\Vert P\right\Vert $ is the number of elements in the
set minus the number of blocks in the partition. For example, $\left\Vert P\right\Vert =1$
iff $P$ identifies only a single pair of elements. It is not hard
to see that $\left\Vert P\right\Vert $ is also the minimal number
of identifications one needs to make in $\mathcal{X}$ in order to
obtain the equivalence relation $P$. 
\begin{defn}
\label{def:imme-quot} Let $\G$ be a core graph and let $P$ be a
partition of $V\left(\G\right)$ with $\left\Vert P\right\Vert =1$,
i.e.\ having a single non-trivial block, of size two. Let $\Delta$
be the core graph generated from $\Gamma$ by $P$. We then say that
$\Delta$ is an \emph{immediate quotient} of $\G$.
\end{defn}
Alternatively, we say that $\Delta$ is \emph{generated by identifying
a single pair} of vertices of $\G$. For instance, the rightmost core
graph in Figure \figref{quotient-graph} is an immediate quotient
of the leftmost one.

The main reason that immediate quotients are interesting is their
algebraic significance. Let $H,J\le\F_{k}$ with $\G=\G_{X}\left(H\right),\Delta=\G_{X}\left(J\right)$
their core graphs, and assume that $\Delta$ is an immediate quotient
of $\G$ obtained by identifying the vertices $u,v\in V\left(\G\right)$.
Now let $w_{u},w_{v}\in\F_{k}$ be the words corresponding to some
paths $p_{u},p_{v}$ in $\G$ from the basepoint to $u$ and $v$
respectively (note that these paths are not unique). It is not hard
to see that identifying $u$ and $v$ has the same effect as adding
the word $w=w_{u}w_{v}^{-1}$ to $H$ and considering the generated
group. Namely, that $J=\langle H,w\rangle$.

\begin{center}
\begin{center}
\xy 
(0,0)*{\otimes}="bp"+(-8,8)*{\G};%
(5,0)*\xycircle(30,15){--};
(24,4)*{\bullet}="u"  +(2,2)*{u};%
(20,-6)*{\bullet}="v" +(2,-2)*{v};%
"bp"; "u" **\crv{(10,20) & (15,-20)} ?(.43)*\dir{>>}; 
(6,8)*{\scriptstyle p_u}; "bp"; "v" **\crv{(5,-10) & (15,10)} ?(.43)*\dir{>>}; 
(6,-6)*{\scriptstyle p_v}; (-60,0)*+{}="dummy";
(0,17)*+{}="dummy";
(0,-17)*+{}="dummy";
\endxy
\par\end{center}
\par\end{center}

The relation of immediate quotients gives the set of finite core graphs
(with edges labeled by $1,\ldots,k$) the structure of a directed
acyclic graph (\emph{DAG})%
\footnote{that is, a directed graph with no directed cycles.%
}. This DAG was first introduced in \cite{Pud14a}, and is denoted
by $\D_{k}$. The set of vertices of $\D_{k}$ consists of the aforementioned
core graphs, and its directed edges connect every core graph to its
immediate quotients. Every ordered basis $X=\left\{ x_{1},\ldots,x_{k}\right\} $
of $\F_{k}$ determines a one-to-one correspondence between the vertices
of this graph and $\mathfrak{sub}_{f\! g}\left(\mathbf{F}_{k}\right)$.

In the case of finite core graphs, $\Delta$ is a quotient of $\G$
if and only if $\Delta$ is reachable from $\G$ in $\D_{k}$ (that
is, there is a directed path from $\Gamma$ to $\Delta$). In other
words, if $H\fg\F_{k}$ then $H\covers J$ iff $\G_{X}\left(J\right)$
can be obtained from $\G_{X}\left(H\right)$ by a finite sequence
of immediate quotients. Thus, for any $H\fg\F_{k}$, the subgraph
of $\D_{k}$ induced by the descendants of $\G_{X}\left(H\right)$
consists of all quotients of $\G_{X}\left(H\right)$, i.e.\ of all
(core graphs corresponding to) elements of $\XF{H}$. By Claim \claref{cover-properties}\enuref{O_X(H)-is-finite},
this subgraph is finite. In Figure \figref{commutator-fringe} we
draw the subgraph of $\D_{k}$ consisting of all quotients of $\G_{X}\left(H\right)$
when $H=\langle x_{1}x_{2}x_{1}^{-1}x_{2}^{-1}\rangle$. The edges
of this subgraph (i.e.\ immediate quotients) are denoted by the dashed
arrows in the figure.

\begin{figure}[h]
\begin{centering}
\includegraphics{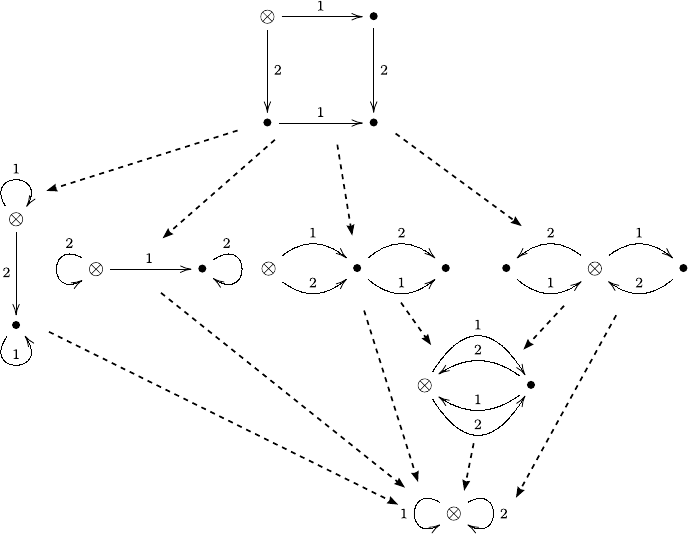}
\par\end{centering}

\caption{The subgraph of $\D_{k}$ induced by $\XF{H}$, that is, all quotients
of the core graph $\G=\G_{X}\left(H\right)$, for $H=\langle x_{1}x_{2}x_{1}^{-1}x_{2}^{-1}\rangle$.
The dashed arrows denote immediate quotients, i.e.\ quotients generated
by merging a single pair of vertices. $\G$ has exactly seven quotients:
itself, four immediate quotients, and two quotients at distance $2$.}

\label{fig:commutator-fringe} 
\end{figure}

It is now natural to define a distance function between a finite core
graph and each of its quotients: 
\begin{defn}
\label{def:distance} Let $H,J\fg\F_{k}$ be subgroups such that $H\covers J$,
and let $\G=\G_{X}\left(H\right)$, $\Delta=\G_{X}\left(J\right)$
be the corresponding core graphs. We define the \emph{$X$-distance}
between $H$ and $J$, denoted $\rho_{X}\left(H,J\right)$ or $\rho\left(\G,\Delta\right)$,
to be the shortest length of a directed path from $\G$ to $\Delta$
in $\D_{k}$. 
\end{defn}
In other words, $\rho_{X}\left(H,J\right)$ is the length of the shortest
series of immediate quotients that yields $\Delta$ from $\Gamma$.
There is another useful equivalent definition for the $X$-distance.
To see this, assume that $\Gamma'$ is generated from $\Gamma$ by
the partition $P$ of $V\left(\Gamma\right)$ and let $\eta:\Gamma\twoheadrightarrow\Gamma'$
be the morphism. For every $x,y\in V\left(\Gamma'\right)$, let $x'\in\eta^{-1}\left(x\right),y'\in\eta^{-1}\left(y\right)$
be arbitrary vertices in the fibers, and let $P'$ be the partition
of $V\left(\Gamma\right)$ obtained from $P$ by identifying $x'$
and $y'$. It is easy to see that the core graph generated from $\Gamma'$
by identifying $x$ and $y$ is the same as the one generated by $P'$
from $\Gamma$. From these considerations we obtain that 
\begin{equation}
\rho_{X}\left(H,J\right)=\min\left\{ \left\Vert P\right\Vert \,\middle|\,{P\mbox{ is a partition of }V\left(\Gamma_{X}\left(H\right)\right)\atop \mbox{such that }\pi_{1}^{X}\left(\nicefrac{\Gamma_{X}\left(H\right)}{P}\right)=J}\right\} .\label{eq:rho_from_partition}
\end{equation}

For example, if $\Delta$ is an immediate quotient of $\Gamma$ then
$\rho_{X}\left(H,J\right)=\rho\left(\Gamma,\Delta\right)=1$. For
$H=\langle x_{1}x_{2}x_{1}^{-1}x_{2}^{-1}\rangle$, $\G_{X}\left(H\right)$
has four quotients at distance $1$ and two at distance $2$ (see
Figure \figref{commutator-fringe}).

As mentioned earlier, by merging a single pair of vertices of $\G_{X}\left(H\right)$
(and then folding) we obtain the core graph of a subgroup $J$ obtained
from $H$ by adding some single generator (thought not every element
of $\F_{k}$ can be added in this manner). Thus, by taking an immediate
quotient, the rank of the associated subgroup increases at most by
1 (in fact, it may also stay unchanged or even decrease). This implies
that whenever $H\covers J$, one has 
\begin{equation}
rk\left(J\right)-rk\left(H\right)~~\le~~\rho_{X}\left(H,J\right)\label{eq:rk-rk_le_rho}
\end{equation}
In \cite{Pud14a} (Lemma 3.3), the distance is bounded from above
as well:
\begin{claim}
\label{cla:bounds_for_rho} Let $H,J\fg\F_{k}$ such that $H\covers J$.
Then 
\[
rk\left(J\right)-rk\left(H\right)~~\le~~\rho_{X}\left(H,J\right)~~\le~~rk\left(J\right)
\]

\end{claim}
We shall make use of the following theorem, which asserts that the
lower bound is attained if and only if $H$ is a free factor of $J$:
\begin{thm}[{\cite[Theorem 1.1]{Pud14a}}]
\label{thm:rho}Let $H,J\fg\F_{k}$ and assume further that $H\covers J$.
Then $H\ff J$ if and only if 
\[
\rho_{X}\left(H,J\right)=\rk\left(J\right)-\rk\left(H\right)
\]

\end{thm}
In fact, the implication which is needed for our proof is trivial:
As mentioned above, merging two vertices in $\G_{X}\left(H\right)$
translates to adding some generator to $H$. If it is possible to
obtain $\G_{X}\left(J\right)$ from $\G_{X}\left(H\right)$ by $rk\left(J\right)-rk\left(H\right)$
merging steps, this means we can obtain $J$ from $H$ by adding $rk\left(J\right)-rk\left(H\right)$
complementary generators to $H$, hence $H\ff J$.%
\footnote{This relies on the well known fact that a set of size $k$ which generates
$\F_{k}$ is a basis. %
} The other implication is not trivial and constitutes the essence
of the proof of Theorem 1.1 in \cite{Pud14a}. The difficulty is that
when $H\ff_{\Xcov}J$, it is not a priori obvious why it is possible
to find $rk\left(J\right)-rk\left(H\right)$ complementing generators
of $J$ from $H$, so that each of them can be realized by merging
a pair of vertices in $\G_{X}\left(H\right)$. \medskip{}

We finish this section with a classical fact about free factors that
will be useful in the next section. 
\begin{claim}
\label{cla:ff-properties}Let $H$, $J$ and $K$ be subgroups of
$\F_{k}$. 
\begin{enumerate}
\item \label{enu:intersection-free-ext}If $H\ff J$ and $K\le J$, then
$H\cap K\ff K$.
\item If $H,K\ff J$ then $H\cap K\ff J$.
\item \label{enu:intermideate-in-free}If $H\ff J$ then $H$ is a free
factor of any intermediate group $H\le M\le J$.
\end{enumerate}
\end{claim}
\begin{proof}
Let $Y$ be a basis of $J$ extending a basis $Y_{0}$ of $H$. Then
$\G_{Y}\left(J\right)$ and $\G_{Y}\left(H\right)$ are bouquets of
$|Y|,|Y_{0}|$ loops, respectively. It is easy to check that $\G_{Y}\left(H\cap K\right)$
is obtained from $\G_{Y}\left(K\right)$ as follows: first, delete
the edges labeled by $Y\setminus Y_{0}$; then, keep only the connected
component of the basepoint; finally, trim all ``hanging trees''
(see the proof of Claim \claref{cover-properties}). Consequently,
$\G_{Y}\left(H\cap K\right)$ is embedded in $\G_{Y}\left(K\right)$.
Claim \claref{morphism-properties}(4) then gives (1), and (2) and
(3) follow immediately.
\end{proof}
In particular, the last claim shows that if $H\ff\F_{k}$ then $\pi\left(H\right)=\infty$
(see Definition \defref{prim_rank}). On the other hand, if $H$ is
not a free factor of $\F_{k}$, then obviously $\pi\left(H\right)\le\rk\left(\F_{k}\right)=k$.
Thus $\pi\left(H\right)\in\left\{ 0,1,2,\ldots,k\right\} \cup\left\{ \infty\right\} $.

\section{\label{sec:Algebraic-Extensions}Algebraic extensions and critical
subgroups}

We now return to the sparsest partial order we consider in this paper,
that of algebraic extensions. All claims in this section appear in
\cite{KM02,MVW07}, except for Lemma \ref{lem:detecting_AE}. We shall
occasionally sketch some proofs in order to allow the reader to obtain
better intuition and in order to exemplify the strength of core graphs. 

Recall (Definition \defref{algebraic-extension}) that $J$ is an
algebraic extension of $H$, denoted $H\alg J$, if $H\le J$ and
$H$ is not contained in any proper free factor of $J$. For example,
consider $H=\left\langle x_{1}x_{2}x_{1}^{\;-1}x_{2}^{\,-1}\right\rangle \le\F_{2}$.
A proper free factor of $\F_{2}$ has rank at most 1, and $H$ is
not contained in any subgroup of rank 1 other than itself (as $x_{1}x_{2}x_{1}^{-1}x_{2}^{-1}$
is not a proper power). Finally, $H$ itself is not a free factor
of $\F_{2}$ (as can be inferred from Theorem \thmref{rho} and Figure
\figref{commutator-fringe}). Thus, $H\alg\F_{2}$. In fact, we shall
see that in this case $\AE{H}=\left\{ H,\F_{2}\right\} $.

We first show that ``$\alg$'' is a partial order:
\begin{claim}
\label{cla:alg-transitive}The relation ``$\alg$'' is transitive.\end{claim}
\begin{proof}
Assume that $H\alg M\alg J.$ Let $H\le L\ff J$. By Claim \claref{ff-properties}\enuref{intersection-free-ext},
$L\cap M\ff M$. But $H\le L\cap M$ and $H\alg M$, so $L\cap M=M$,
and thus $M\le L.$ So now $M\le L\ff J$, and from $M\alg J$ we
obtain that $L=J$.
\end{proof}
Next, we show that ``$\alg$'' is dominated by ``$\covers$''
for every basis $X$ of $\F_{k}$. Namely, if $H\alg J$ then $H\covers J$.
This shows, in particular, that the poset $\left(\mathfrak{sub}_{f\! g}\left(\mathbf{F}_{k}\right),\alg\right)$
is locally-finite.
\begin{claim}
\label{cla:alg-then-covers}If $H\alg J$ then $H\covers J$ for every
basis $X$ of $\F_{k}$.\end{claim}
\begin{proof}
By Claim \claref{cover-properties}, there is an intermediate subgroup
$M$ such that $H\covers M\ff J$, and from $H\alg J$ it follows
that $M=J$.\end{proof}
\begin{rem}
\label{rem:alg-cover-inc}It is natural to conjecture that the converse
also holds, namely that if $H\covers J$ for every basis $X$ of $\F_{k}$
then $H\alg J$. (In fact, this conjecture appears in \cite{MVW07},
Section 3.) This is, however, false: it turns out that for $H=\left\langle x_{1}^{\,2}x_{2}^{\,2}\right\rangle $
and $J=\left\langle x_{1}^{\,2}x_{2}^{\,2},x_{1}x_{2}\right\rangle $,
$H\covers J$ for every basis $X$ of $\F_{2}$, but $J$ is \emph{not}
an algebraic extension of $H$ \cite{PP14}. However, there are bases
of $\F_{3}$ with respect to which $H$ does not cover $J$. Hence,
it is still plausible that some weaker version of the conjecture holds,
e.g.\ that $H\alg J$ if and only if for every embedding of $J$
in a free group $F$, and for every basis $X$ of $F$, $H\covers J$.
It is also plausible that the original conjecture from \cite{MVW07}
holds for $\F_{k}$ with $k\ge3$. \\
In a similar fashion, one can ask whether $H\le J$ if and only if
for some basis $X$ of $\F_{k}$, $H\covers J$.
\end{rem}
Claim \ref{cla:alg-then-covers} completes the proof of the relations,
mentioned in Section \secref{Overview-of-the}, between the different
partial orders we consider in this paper: inclusion, the family $\covers$,
and algebraic extensions. Recall that $H$-critical subgroups are
a special kind of algebraic extensions. Thus: 
\[
\crit\left(H\right)\subseteq\AE{H}\subseteq\XF{H}\subseteq[H,\infty)_{\le}.
\]
Theorem \thmref{rho} and Claim \claref{alg-then-covers} give the
following criterion for algebraic extensions:
\begin{lem}
\label{lem:detecting_AE}Let $H\fg\F_{k}$. The algebraic extensions
of $H$ are the elements of $\XF{H}$ which are \emph{not} immediate
quotients of any subgroup in $\XF{H}$ of smaller rank.\end{lem}
\begin{proof}
Let $J\in\XF{H}$. If $J$ is an immediate $X$-quotient of $L\in\XF{H}$
with $\rk\left(L\right)<\rk\left(J\right)$, then by Theorem \thmref{rho}
$H\le L\pff J$, hence $J$ is not an algebraic extension of $H$.
On the other hand, assume there exists some $L$ such that $H\le L\pff J$.
By Claim \claref{cover-properties}(1), there exists $M$ such that
$H\covers M\ff L\pff J$. By Claim \claref{cover-properties}(3),
$M\pff_{\Xcov}J$. From Theorem \thmref{rho} it follows that there
is a chain of immediate quotients $M=M_{0}\leq M_{1}\leq\ldots\leq M_{r}=J$
inside $\XF{H}$ with $\rk\left(M_{i+1}\right)=\rk\left(M_{i}\right)+1$,
and $M_{r-1}$ is the group we have looked for.
\end{proof}
Since the subgraph of $\D_{k}$ induced by the vertices corresponding
to $\XF{H}$, namely $\G_{X}\left(H\right)$ and its descendants,
is finite and can be effectively computed, Lemma \lemref{detecting_AE}
yields a straight-forward algorithm to find all algebraic extensions
of a given $H\fg\F_{k}$ (this algorithm was first introduced in \cite{Pud14a}).
This, in particular, allows one to find all $H$-critical subgroups,
and thus to compute the primitivity rank $\pi\left(H\right)$: the
subgroups constituting $\crit\left(H\right)$ are those in $\left(H,\infty\right)_{alg}$
of minimal rank, which is $\pi\left(H\right)$. For instance, Figure
\figref{commutator-fringe} shows that for $H=\left\langle x_{1}x_{2}x_{1}^{-1}x_{2}^{-1}\right\rangle $
we have ${H}=\left\{ H,\F_{2}\right\} $. Thus, $\crit\left(H\right)=\left\{ \F_{2}\right\} $
and $\pi\left(H\right)=2$ (so $\rp\left(H\right)=1$). 

\medskip{}

We conclude this section with yet another elegant result from \cite{KM02,MVW07}
that will be used in the proof of Theorem \thmref{avg_fixed_points}.
In the spirit of field extensions, it says that every extension of
subgroups of $\F_{k}$ has a unique factorization to an algebraic
extension followed by a free extension (compare this with Claim \claref{cover-properties}(1,2)):
\begin{claim}
\label{cla:alg-closure}Let $H\le J$ be free groups. Then there is
a unique subgroup $L$ of $J$ such that $H\alg L\ff J$. Moreover,
$L$ is the intersection of all intermediate free factors of $J$
and the union of all intermediate algebraic extensions of $H$:
\begin{equation}
L=\bigcap_{M:\, H\le M\ff J}M=\bigcup_{M:\, H\alg M\leq J\vphantom{\ff}}M\label{eq:min_ff_max_alg}
\end{equation}

\end{claim}
In particular, the intersection of all free factors is a free factor,
and the union of all algebraic extensions is an algebraic extension.
Claim \claref{alg-closure} is true in general, but we describe the
proof only of the slightly simpler case of finitely generated subgroups.
We need only this case in this paper.
\begin{proof}
By Claim \claref{ff-properties} and rank considerations, the intersection
in the middle of \eqref{min_ff_max_alg} is by itself a free factor
of $J$. Denote it by $L$, so we have $H\le L\ff J$. Clearly, $L$
is an algebraic extension of $H$ (otherwise it would contain a proper
free factor). But we claim that $L$ contains every other intermediate
algebraic extension of $H$. Indeed, let $H\alg M\le J$. By Claim
\claref{ff-properties}\enuref{intersection-free-ext}, $H\le M\cap L\ff M$,
so $M\cap L=M$, that is $M\le L$.
\end{proof}

\section{\label{sec:Mobius-Inversions}Möbius inversions}

Let $\left(P,\le\right)$ be a locally-finite poset and let $A$ be
a commutative ring with unity. Then there exists an \emph{incidence
algebra}%
\footnote{The theory of incidence algebras of posets can be found in \cite{Stan97}.%
} of all functions from pairs $\left\{ \left(x,y\right)\in P\times P\,\middle|\, x\leq y\right\} $
to $A$. In addition to point-wise addition and scalar multiplication,
it has an associative multiplication defined by convolution:

\[
(f*g)(x,y)=\sum_{z\in[x,y]}f(x,z)g(z,y)
\]
 (where $x\le y$ and $\left[x,y\right]=\left\{ z\,\middle|\, x\leq z\leq y\right\} $).
The unit element is the diagonal
\[
\delta(x,y)=\begin{cases}
1 & x=y\\
0 & x\lneqq y
\end{cases}.
\]
Functions with invertible diagonal entries (i.e.\ $f\left(x,x\right)\in A^{\times}$
for all $ $$x\in P$) are invertible w.r.t.\ this multiplication.
Most famously, the constant $\zeta$ function, which is defined by
$\zeta\left(x,y\right)=1$ for all $x\leq y$, is invertible, and
its inverse, $\mu$, is called the\emph{ Möbius function} of $P$.
This means that $\zeta*\mu=\mu*\zeta=\delta$, i.e., for every pair
$x\le y$ 
\[
\sum_{z\in[x,y]}\mu(z,y)=(\zeta*\mu)(x,y)=\delta(x,y)=(\mu*\zeta)(x,y)=\sum_{z\in[x,y]}\mu(x,z).
\]

Let $f$ be some function in the incidence algebra. The function $f*\zeta$,
which satisfies $\left(f*\zeta\right)\left(y\right)=\sum_{z\in\left[x,y\right]}f\left(z\right)$,
is analogous to the right-accumulating function in calculus (for $g:\mathbb{R}\rightarrow\mathbb{R}$
this is the function $G\left(y\right)=\int_{z\in\left[x,y\right]}g\left(z\right)dz$).
Thus, multiplying a function on the right by $\mu$ can be thought
of as ``right derivation''. Similarly, one thinks of multiplying
from the left by $\zeta$ and $\mu$ as left integration and left
derivation, respectively.

Recall the function $\Phi$ \eqref{Phi}, defined for every pair of
free subgroups $H,J\fg\F_{k}$ such that $H\le J$: $\Phi_{H,J}\left(n\right)$
is the expected number of common fixed points of $\alpha_{J,n}\left(H\right)$,
where $\alpha_{J,n}\in\Hom\left(J,S_{n}\right)$ is a random homomorphism
chosen with uniform distribution. We think of $\Phi$ as a function
from the set of such pairs $\left(H,J\right)$ into the ring of functions
$\mathbb{N}\to\mathbb{Q}$. 

Let $X$ be a basis of $\F_{k}$. We write $\Phi^{X}$ for the restriction
of $\Phi$ to pairs $\left(H,J\right)$ such that $H\covers J$. As
``$\covers$'' defines a locally finite partial ordering of $\mathfrak{sub}_{f\! g}\left(\mathbf{F}_{k}\right)$,
there exists a matching Möbius function, $\mu^{X}=\left(\zeta^{X}\right)^{-1}$
(where $\zeta_{H,J}^{X}=1$ for all $H\covers J$). Our proof of Theorem
\thmref{avg_fixed_points} consists of a detailed analysis of the
left, right, and two-sided derivations of $\Phi^{X}$:
\[
\xymatrix{ & \Phi^{X}\ar@{-}[dl]\ar@{-}[dr]\\
L^{X}\overset{{\scriptscriptstyle {def}}}{=}\mu^{X}*\Phi^{X}\ar@{-}[dr] &  & R^{X}\overset{{\scriptscriptstyle {def}}}{=}\Phi^{X}*\mu^{X}\ar@{-}[dl]\\
 & C^{X}\overset{{\scriptscriptstyle {def}}}{=}\mu^{X}*\Phi^{X}*\mu^{X}
}
\]
By definition, we have for every f.g.\ $H\covers J$:
\begin{equation}
\Phi_{H,J}=\negthickspace\sum_{M\in\XC{H}{J}}\negthickspace L_{M,J}^{X}=\negthickspace\sum_{M,N:\, H\covers M\covers N\covers J}\negthickspace C_{M,N}^{X}=\negthickspace\sum_{N\in\XC{H}{J}}\negthickspace R_{H,N}^{X}\label{eq:mobius}
\end{equation}
Note that \eqref{mobius} can serve as definitions for the three functions
$L^{X},C^{X},R^{X}$: for instance, $L^{X}=\mu^{X}*\Phi^{X}$ is equivalent
to $\zeta^{X}*L^{X}=\Phi^{X}$, which is the leftmost equality above.

We begin the analysis of these functions by the following striking
observation regarding $R^{X}$. Recall (Claim \claref{alg-then-covers})
that if $H\alg J$ then $H\covers J$ for every basis $X$. It turns
out that the function $R^{X}$ is supported on algebraic extensions
alone, and moreover, is independent of the basis $X$.
\begin{prop}
\label{prop:R-independent-of-basis}Let $H,J\fg\F_{k}$.
\begin{enumerate}
\item If $H\covers J$ but $J$ is \emph{not} an algebraic extension of
$H$, then $R_{H,J}^{X}=0$.
\item $R_{H,J}^{X}=R_{H,J}^{Y}$ for every basis $Y$ of $\F_{k}$, whenever
both are defined.
\end{enumerate}
\end{prop}
\begin{rem}
The only property of $\Phi$ we use is that $\Phi_{H,L}=\Phi_{H,J}$
whenever $H\leq L\ff J$, which is easy to see from the definition
of $\Phi$. Therefore, the proposition holds for the right derivation
of every function with this property. In particular, the proposition
holds for every ``statistical'' function, in which the value of
$\left(H,J\right)$ depends solely on the image of $H$ via a uniformly
distributed random homomorphism from $J$ to some group $G$. \end{rem}
\begin{proof}
We show both claims at once by induction on $\left|\XC{H}{J}\right|$,
the size of the closed interval between $H$ and $J$. The induction
basis is $H=J$. That $H\alg H$ is immediate. By \eqref{mobius},
$R_{H,H}^{X}=\Phi_{H,H}$ and so $R_{H,H}^{X}$ is indeed independent
of the basis $X$. 

Assume now that $\left|\XC{H}{J}\right|=r$ and that both claims are
proven for every pair bounding an interval of size $<r$. By \eqref{mobius}
and the first claim of the induction hypothesis, 
\begin{equation}
R_{H,J}^{X}=\Phi_{H,J}-\negthickspace\sum_{N\in\XCO{H}{J}}\negthickspace R_{H,N}^{X}=\Phi_{H,J}-\negthickspace\sum_{N:\, H\alg N\ncovers J}\negthickspace R_{H,N}^{X}\label{eq:RHJ}
\end{equation}
By Claim \claref{cover-properties}\enuref{Cover-insideCover}, $\left\{ N\,\middle|\, H\alg N\ncovers J\right\} =\left\{ N\,\middle|\, H\alg N\lneqq J\right\} $,
and the latter is independent of the basis $X$. Furthermore, by the
induction hypothesis regarding the second claim, so are the terms
$R_{H,N}^{X}$ in this summation. This settles the second point.

Finally, if $J$ is \emph{not} an algebraic extension of $H$ then
let $L$ be some intermediate free factor of $J$, $H\le L\pff J$.
As mentioned above, this yields that $\Phi_{H,J}=\Phi_{H,L}$. Therefore,
\begin{align*}
R_{H,J}^{X} & =\Phi_{H,J}-\negthickspace\sum_{N\in\XCO{H}{J}}\negthickspace R_{H,N}^{X}=\underbrace{\Phi_{H,L}-\negthickspace\sum_{N\in\XC{H}{L}}\negthickspace R_{H,N}^{X}}_{0\,\mathrm{by\, definition}}-\sum_{N\in\XCO{H}{J}\setminus\XC{H}{L}}\negthickspace R_{H,N}^{X}
\end{align*}
By Claim \claref{alg-closure}, all algebraic extensions of $H$ inside
the interval $\XC{H}{J}$ are contained in $L$. Hence, every subgroup
$N\in\XCO{H}{J}\setminus\XC{H}{L}$ is not an algebraic extension
of $H$, and by the induction hypothesis $R_{H,N}^{X}$ vanishes.
The desired result follows.
\end{proof}
In view of Proposition \propref{R-independent-of-basis} we can omit
the superscript and write from now on $R_{H,J}$ instead of $R_{H,J}^{X}$.
Moreover, we can write the following ``basis independent'' equation
for every pair of f.g.\ subgroups $H\le J$:
\begin{equation}
\Phi_{H,J}=\sum_{N:\, H\alg N\le J}R_{H,N}.\label{eq:phi=00003Dsum-of-alg}
\end{equation}
When $H\covers J$ this follows from the proof above. For general
$H\le J$, there is some subgroup $L$ such that $H\covers L\ff J$
and every intermediate algebraic extension $H\alg N\leq J$ is contained
in $L$ (see Claims \claref{cover-properties} and \claref{alg-closure}).
Therefore, 
\[
\Phi_{H,J}=\Phi_{H,L}=\negthickspace\sum_{N:\, H\alg N\le L}\negthickspace R_{H,N}=\negthickspace\sum_{N:\, H\alg N\le J}\negthickspace R_{H,N}.
\]

It turns out that unlike the function $R$, the other two derivations
of $\Phi$, namely $L^{X}$ and $C^{X}$, do depend on the basis $X$.
However, the latter two functions have combinatorial interpretations.
In the next section we show that $\Phi_{H,J}$ and $L_{H,J}^{X}$
can be described in terms of random coverings of the core graph $\G_{X}\left(J\right)$,
and that explicit rational expressions in $n$ can be computed to
express these two functions for given $H,J$ (Lemmas \lemref{Phi-as-lifts}
and \lemref{L-as-lifts} below). This, in turn, allows us to analyze
the combinatorial meaning and order of magnitude of $C_{M,N}^{X}$
(Proposition \propref{order-magnitude-C}). 

Finally, using the fact that $R$ is the ``left integral'' of $C^{X}$,
that is $R=\zeta^{X}*C^{X}$, we finish the circle around the diagram
of $\Phi$'s derivations, and use this analysis of $\Phi$, $L^{X}$
and $C^{X}$ to prove that for every pair $H\alg J$, $R_{H,J}$ does
not vanish and is, in fact, positive for large enough $n$. This alone
gives Theorem \thmref{mp_is_prim}. The more informative \thmref{avg_fixed_points}
follows from an analysis of the order of magnitude of $R_{H,J}$ in
this case (Proposition \propref{order-magnitude-R}).

\section{\label{sec:Random-Covers}Random coverings of core graphs}

This section studies the graphs which cover a given core-graph in
the topological sense, i.e.\ $\widehat{\Gamma}\overset{p}{\twoheadrightarrow}\Gamma$
with $p$ locally bijective. We call these graphs (together with their
projection maps) \emph{coverings} of $\Gamma$. The reader should
not confuse this with our notion ``covers'' from Definition \defref{quotient}. 

We focus on directed and edge-labeled coverings. This means we only
consider $\widehat{\Gamma}\overset{p}{\twoheadrightarrow}\Gamma$
such that $\widehat{\Gamma}$ is directed and edge-labeled, and the
projection $p$ preserves orientations and labels. When $\Gamma$
is a core-graph we do \emph{not} assume that $\widehat{\Gamma}$ is
a core-graph as well. It may be disconnected, and it need not be pointed.
Nevertheless, it is not hard to see that when $\Gamma$ and $\widehat{\Gamma}$
are finite, for every vertex $v$ in $p^{-1}\left(\otimes\right)$,
the fiber over $\Gamma$'s basepoint, we do have a valid core-graph,
which we denote by $\widehat{\Gamma}_{v}$: this is the connected
component of $v$ in $\widehat{\Gamma}$, with $v$ serving as basepoint.
Moreover, the restriction of the projection map $p$ to $\widehat{\Gamma}_{v}$
is a core-graph morphism. 

The theory of core-graph coverings shares many similarities with the
theory of topological covering spaces. The following claim lists some
standard properties of covering spaces, formulated for core-graphs.
\begin{claim}
\label{cla:core-coverings}Let $\Gamma$ be a core-graph, $\widehat{\Gamma}\overset{p}{\twoheadrightarrow}\Gamma$
a covering and $v$ a vertex in the fiber $p^{-1}\left(\otimes\right)$. 
\begin{enumerate}
\item The group $\pi_{1}^{X}\left(\Gamma\right)$ acts on the fiber $p^{-1}\left(\otimes\right)$,
and these actions give a correspondence between coverings of $\Gamma$
and $\pi_{1}^{X}\left(\Gamma\right)$-sets.
\item \label{enu:n-coverings-morphisms}In this correspondence, coverings
of $\Gamma$ with fiber $\left\{ 1,\ldots,n\right\} $ correspond
to actions of $\pi_{1}^{X}\left(\Gamma\right)$ on $\left\{ 1,\ldots,n\right\} $,
i.e., to group homomorphisms $\pi_{1}^{X}\left(\Gamma\right)\rightarrow S_{n}$.
\item The group $\pi_{1}^{X}\left(\widehat{\Gamma}_{v}\right)$ is the stabilizer
of $v$ in the action of $\pi_{1}^{X}\left(\Gamma\right)$ on $p^{-1}\left(\otimes\right)$
(note that $\pi_{1}^{X}\left(\widehat{\Gamma}_{v}\right)$ and $\pi_{1}^{X}\left(\Gamma\right)$
are both subgroups of $\F\left(X\right)$).
\item \label{enu:covering-lifting}A core-graph morphism $\Delta\rightarrow\Gamma$
can be lifted to a core-graph morphism $\Delta\rightarrow\widehat{\Gamma}_{v}$
(i.e., the diagram 
\[
\xymatrix{ & \widehat{\Gamma}_{v}\ar@{->>}[d]^{p}\\
\Delta\ar[r]\ar@{-->}[ur] & \Gamma
}
\]
can be completed) if and only if $\pi_{1}^{X}\left(\Delta\right)\subseteq\pi_{1}^{X}\left(\widehat{\Gamma}_{v}\right)$.
By the previous point, this is equivalent to saying that all elements
of $\pi_{1}^{X}\left(\Delta\right)$ fix $v$.
\end{enumerate}
\end{claim}
We now turn our attention to random coverings. The vertex set of an
$n$-sheeted covering of a graph $\G=\left(V,E\right)$ can be assumed
to be $V\times\left\{ 1,\ldots,n\right\} $, so that the fiber above
$v\in V$ is $\left\{ v\right\} \times\left\{ 1,\ldots,n\right\} $.
For every edge $e=\left(u,v\right)\in E$, the fiber over $e$ then
constitutes a perfect matching between $\left\{ v\right\} \times\left\{ 1,\ldots,n\right\} $
and $\left\{ u\right\} \times\left\{ 1,\ldots,n\right\} $. This suggests
a natural model for random $n$-coverings of the graph $\G$. Namely,
for every $e\in E$ choose uniformly a random perfect matching (which
is just a permutation in $S_{n}$). This model was introduced in \cite{AL02},
and is a generalization of a well-known model for random regular graphs
(see e.g.\ \cite{BS87}).%
\footnote{Occasionally these random coverings are referred to as random \emph{lifts}
of graphs. We shall reserve this term for its usual meaning.%
} Note that the model works equally well for graphs with loops and
with multiple edges.

In fact, there is some redundancy in this model, if we are interested
only in isomorphism classes of coverings (two coverings are isomorphic
if there is an isomorphism between them that commutes with the projection
maps). It is possible to obtain the same distribution on (isomorphism
classes of) $n$-coverings of $\G$ with fewer random permutations:
one may choose some spanning tree $T$ of $\G$, associate the identity
permutation with every edge in $T$, and pick random permutations
only for edges outside $T$. 

We now fix some $J\fg\F_{k}$, and consider random coverings of its
core-graph, $\G_{X}\left(J\right)$. We denote by $\cJ$ a random
$n$-covering of $\G_{X}\left(J\right)$, according to one of the
models described above. If $p:\cJ\to\G_{X}\left(J\right)$ is the
covering map, then $\cJ$ inherits the edge orientation and labeling
from $\G_{X}\left(J\right)$ via $p^{-1}$. For every $i$ ($1\le i\le n)$,
we write $\cJ_{i}$ for the core-graph $\cJ_{\left(\otimes,i\right)}$
(the component of $\left(\otimes,i\right)$ in $\cJ$ with basepoint
$\left(\otimes,i\right)$). 

By Claim \claref{core-coverings}\enuref{n-coverings-morphisms},
each random $n$-covering of $\G_{X}\left(J\right)$ encodes a homomorphism
$\alpha_{J,n}\in\Hom\left(J,S_{n}\right)$, via the action of $J=\pi_{1}^{X}\left(\Gamma_{X}\left(J\right)\right)$
on the basepoint fiber. Explicitly, an element $w\in J$ is mapped
to a permutation $\alpha_{J,n}\left(w\right)\in S_{n}$ as follows:
$w$ corresponds to a closed path $p_{w}$ around the basepoint of
$\G_{X}\left(J\right)$. For every $1\le i\le n$, the lift of $p_{w}$
that starts at $\left(\otimes,i\right)$ ends at $\left(\otimes,j\right)$
for some $j$, and $\alpha_{J,n}\left(w\right)\left(i\right)=j$. 

By the correspondence of actions of $J$ on $\left\{ 1,\ldots,n\right\} $
and $n$-coverings of $\Gamma_{X}\left(J\right)$, $\alpha_{J,n}$
is a uniform random homomorphism in $\Hom\left(J,S_{n}\right)$. This
can also be verified using the ``economical'' model, as follows:
choose some basis $Y=\left\{ y_{1},\ldots,y_{\rk\left(J\right)}\right\} $
for $J$ via a choice of a spanning tree $T$ of $\G_{X}\left(J\right)$
and of orientation of the remaining edges, and choose uniformly at
random some $\sigma_{r}\in S_{n}$ for every basis element $y_{r}$.
Clearly, $\alpha_{J,n}\left(y_{r}\right)=\sigma_{r}$.\medskip{}

We can now use the coverings of $\Gamma_{X}\left(J\right)$ to obtain
a geometric interpretation of $\Phi_{H,J}$, as follows: let $H\le J\fg\F_{k}$
and $1\leq i\leq n$. By \claref{core-coverings}\enuref{covering-lifting},
the morphism $\eta_{H\rightarrow J}^{X}:\Gamma_{X}\left(H\right)\rightarrow\Gamma_{X}\left(J\right)$
lifts to a core-graph morphism $\Gamma_{X}\left(H\right)\rightarrow\widehat{\Gamma}_{X}\left(J\right)_{i}$
iff $H=\pi_{1}^{X}\left(\Gamma_{X}\left(H\right)\right)$ fixes $\left(\otimes,i\right)$
via the action of $J$ on the fiber $\otimes\times\left\{ 1,\ldots,n\right\} $.
Since this action is given by $\alpha_{J,n}$, this means that $\eta_{H\rightarrow J}^{X}$
lifts to $\widehat{\Gamma}_{X}\left(J\right)_{i}$ exactly when $\alpha_{J,n}\left(H\right)$
fixes $i$. Recalling that $\Phi_{H,J}\left(n\right)$ is the expected
number of elements in $\left\{ 1,\ldots,n\right\} $ fixed by $\alpha_{J,n}\left(H\right)$,
we obtain an alternative definition for it:
\begin{lem}
\label{lem:Phi-as-lifts} Let $\cJ$ be a random $n$-covering space
of $\G_{X}\left(J\right)$ in the aforementioned model from \cite{AL02}.
Then, 
\[
\Phi_{H,J}\left(n\right)=\textrm{The expected number of lifts of \ensuremath{\eta_{H\rightarrow J}^{X}}\,\ to \ensuremath{\cJ}}.
\]
\[
\xymatrix{ & \cJ\ar@{->>}[d]^{p}\\
\Gamma_{X}\left(H\right)\ar[r]_{\eta_{H\rightarrow J}^{X}}\ar@{-->}[ur] & \Gamma_{X}\left(J\right)
}
\]

\end{lem}
Note that this characterization of $\Phi_{H,J}$ involves the basis
$X$, although the original definition \eqref{Phi} does not. One
of the corollaries of this lemma is therefore that the average number
of lifts does \emph{not} depend on the basis $X$.

Recall (Section \secref{Mobius-Inversions}) the definition of the
function $L^{X}$, which satisfies $\Phi_{H,J}=\sum_{M\in\XC{H}{J}}L_{M,J}^{X}$
for every $H\covers J$. It turns out that this derivation of $\Phi$
also has a geometrical interpretation. Assume that $\eta_{H\rightarrow J}^{X}$
does lift to $\widehat{\eta}_{i}:\G_{X}\left(H\right)\to\cJ_{i}$.
By Claim \claref{cover-properties}, $\widehat{\eta}_{i}$ decomposes
as a quotient onto $\Gamma_{X}\left(M\right)$, where $M=\pi_{1}^{X}\left(\mathrm{im}\,\widehat{\eta}_{i}\right)$,
followed by an embedding. Moreover, $M$ lies in $\XC{H}{J}$. On
the other hand, if there is some $M\in\XC{H}{J}$ such that $\G_{X}\left(M\right)$
is embedded in $\cJ_{i}$ then such $M$ is unique and $\widehat{\eta}_{i}$
lifts to the composition of $\eta_{H\rightarrow M}^{X}$ with this
embedding. Consequently, 
\begin{align*}
\Phi_{H,J}\left(n\right) & =\textrm{Expected number of lifts of \ensuremath{\eta_{H\to J}^{X}}\,\ to \ensuremath{\cJ}}\\
 & =\negthickspace\sum_{M\in\XC{H}{J}}\negthickspace\negthickspace\textrm{Expected number of \emph{injective} lifts of \ensuremath{\eta_{M\to J}^{X}}\,\ to \ensuremath{\cJ}.}
\end{align*}
Taking the left derivations, we obtain:
\begin{lem}
\label{lem:L-as-lifts}Let $M\covers J$, and let $\cJ$ be a random
$n$-covering space of $\G_{X}\left(J\right)$ in the aforementioned
model from \cite{AL02}. Then, 
\[
L_{M,J}^{X}\left(n\right)=\textrm{The expected number of \emph{injective} lifts of \ensuremath{\eta_{M\to J}^{X}}\,\ to \ensuremath{\cJ}}.
\]
\[
\xymatrix{ & \cJ\ar@{->>}[d]^{p}\\
\Gamma_{X}\left(M\right)\ar[r]_{\eta_{M\rightarrow J}^{X}}\ar@{^{(}-->}[ur] & \Gamma_{X}\left(J\right)
}
\]

\end{lem}
Unlike the number of lifts in general, the number of injective lifts
does depend on the basis $X$. For instance, consider $M=\la x_{1}x_{2}\ra$
and $J=\la x_{1},x_{2}\ra=\F_{2}$. With the basis $X=\left\{ x_{1},x_{2}\right\} $,
the probability that $\eta_{M\rightarrow J}^{X}$ lifts injectively
to $\cJ_{i}$ equals $\frac{n-1}{n^{2}}$ (Lemma \lemref{L-formula}
shows how to compute this). However, with the basis $Y=\left\{ x_{1}x_{2},x_{2}\right\} $,
the corresponding probability is $\frac{1}{n}$. We also remark that
Lemma \lemref{L-as-lifts} allows a natural extension of $L^{X}$
to pairs $M,J$ such that $M$ does not $X$-cover $J$.

Lemma \lemref{L-as-lifts} allows us to generalize the method used
in \cite{Nic94,LP10,Pud14a} to compute the expected number of fixed
points in $\alpha_{n}\left(w\right)$ (see the notations before Theorem
\nameref{thm:mp_is_prim_ext}). We claim that for $n$ large enough,
$L_{M,J}^{X}\left(n\right)$ is a simple rational expression in $n$. 
\begin{lem}
\label{lem:L-formula}Let $M,J\fg\F_{k}$ such that $M\covers J$,
and let $\eta=\eta_{M\to J}^{X}$ be the core-graph morphism. For
large enough $n$,
\begin{equation}
L_{M,J}^{X}\left(n\right)=\frac{\prod\limits _{v\in V\left(\G_{X}\left(J\right)\right)}\left(n\right)_{\left|\eta^{-1}(v)\right|}}{\prod\limits _{e\in E\left(\G_{X}\left(J\right)\right)}\left(n\right)_{\left|\eta^{-1}(e)\right|}},\label{eq:L-formula}
\end{equation}
where $\left(n\right)_{r}$ is the falling factorial $n\left(n-1\right)\ldots\left(n-r+1\right)$,
and ``large enough $n$'' is $n\ge\max\limits _{e\in E\left(\G_{X}\left(J\right)\right)}\left|\eta^{-1}\left(e\right)\right|$
(so that the denominator does not vanish). \end{lem}
\begin{proof}
Let $v$ be a vertex in $\G_{X}\left(J\right)$ and consider the fiber
$\eta^{-1}\left(v\right)$ in $\G_{X}\left(M\right)$. For every injective
lift $\widehat{\eta}:\G_{X}\left(M\right)\hookrightarrow\cJ$, the
fiber $\eta^{-1}\left(v\right)$ is mapped injectively into the fiber
$p^{-1}\left(v\right)$. The number of such injections is 
\[
\left(n\right)_{|\eta^{-1}(v)|}=n(n-1)\ldots(n-|\eta^{-1}(v)|+1),
\]
and therefore the number of injective lifts of $\eta\big|_{V\left(\Gamma_{X}\left(M\right)\right)}$
into $V\left(\cJ\right)$ is the numerator of \eqref{L-formula}.

We claim that any such injective lift has a positive probability of
extending to a full lift of $\eta$: all one needs is that the fiber
above every edge of $\G_{X}\left(J\right)$ satisfy some constraints.
To get the exact probability, we return to the more ``wasteful''
version of the model for a random $n$-covering of $\G_{X}\left(J\right)$,
the model in which we choose a random permutation for every edge of
the base graph. Let $\widehat{\eta}:V\left(\G_{X}\left(M\right)\right)\hookrightarrow V\left(\cJ\right)$
be an injective lift of the vertices of $\G_{X}\left(M\right)$ as
above, and let $e$ be some edge of $\G_{X}\left(J\right)$. If $\widehat{\eta}$
is to be extended to $\eta^{-1}\left(e\right)$, the fiber above $e$
in $\cJ$ must contain, for every $\left(u,v\right)\in\eta^{-1}\left(e\right)$,
the edge $\left(\widehat{\eta}\left(u\right),\widehat{\eta}\left(v\right)\right).$

Thus, the random permutation $\sigma\in S_{n}$ which determines the
perfect matching above $e$ in $\cJ$, must satisfy $\left|\eta^{-1}\left(e\right)\right|$
non-colliding constraints of the form $\sigma\left(i\right)=j$. Whenever
$n\geq\left|\eta^{-1}\left(e\right)\right|$ (which we assume), a
uniformly random permutation in $S_{n}$ satisfies such constraints
with probability
\[
\frac{1}{\left(n\right)_{\left|\eta^{-1}(e)\right|}}.
\]
This shows the validity of \eqref{L-formula}.
\end{proof}
This immediately gives a formula for $\Phi_{H,J}$ as a rational function:
\begin{cor}
\label{cor:Phi-formula}Let $H,J\fg\F_{k}$ such that $H\covers J$.
Then, for large enough $n$,
\[
\Phi_{H,J}\left(n\right)=\negthickspace\sum_{M\in\XC{H}{J}}\negthickspace L_{M,J}^{X}\left(n\right)=\sum_{M\in\XC{H}{J}}\frac{\prod\limits _{v\in V\left(\G_{X}\left(J\right)\right)}\left(n\right)_{\left|\left(\eta_{M\to J}^{X}\right)^{-1}(v)\right|}}{\prod\limits _{e\in E\left(\G_{X}\left(J\right)\right)}\left(n\right)_{\left|\left(\eta_{M\to J}^{X}\right)^{-1}(e)\right|}}.
\]

\end{cor}
Since $H$ $X$-covers every intermediate $M\in\XC{H}{J}$, the largest
fiber above every edge of $\G_{X}\left(J\right)$ is obtained in $\G_{X}\left(H\right)$
itself. Thus, ``large enough $n$'' in this Corollary can be replaced
by $n\ge\max\limits _{e\in E\left(\G_{X}\left(J\right)\right)}\left|\left(\eta_{H\to J}^{X}\right)^{-1}\left(e\right)\right|$.

In fact, Corollary \ref{cor:Phi-formula} applies, with slight modifications,
to every pair of f.g. subgroups $H\le J$: Lemma \ref{lem:Phi-as-lifts}
holds in this more general case, that is $\Phi_{H,J}$ is equal to
the expected number of lifts of $\Gamma_{X}\left(H\right)$ to the
random $n$-covering $\cJ$. The image of each lift (with the image
of $\otimes$ as basepoint) is a core graph which is a quotient of
$\Gamma_{X}\left(H\right)$, and so corresponds to a subgroup $M$
such that $H\covers M\le J$. In explaining the rational expression
in Lemma \ref{lem:L-formula} we did not need $M$ to cover $J$.
Thus, for every $H\le J$, both finitely generated, 
\begin{equation}
\Phi_{H,J}\left(n\right)=\sum_{M\,:\, H\covers M\le J}\frac{\prod\limits _{v\in V\left(\G_{X}\left(J\right)\right)}\left(n\right)_{\left|\left(\eta_{M\to J}^{X}\right)^{-1}(v)\right|}}{\prod\limits _{e\in E\left(\G_{X}\left(J\right)\right)}\left(n\right)_{\left|\left(\eta_{M\to J}^{X}\right)^{-1}(e)\right|}}.\label{eq:phi-formula-general}
\end{equation}

Corollary \corref{Phi-formula} yields in particular a straight-forward
algorithm to obtain a rational expression in $n$ for $\Phi_{H,J}\left(n\right)$
(valid for large enough $n$). For example, consider $H=\left\langle x_{1}x_{2}x_{1}^{-1}x_{2}^{-1}\right\rangle $
and $\F_{2}=\left\langle x_{1},x_{2}\right\rangle $. The interval
$\XC{H}{\F_{2}}$ consists of seven subgroups, as depicted in Figure
\figref{commutator-fringe}. Following the computation in Corollary
\corref{Phi-formula}, we get that for $n\ge2$ (we scan the quotients
in Figure \figref{commutator-fringe} top-to-bottom and in each row
left-to-right):
\begin{align*}
\Phi_{H,\F_{2}}\left(n\right)= & \:\frac{\left(n\right)_{4}}{\left(n\right)_{2}\left(n\right)_{2}}+\frac{\left(n\right)_{2}}{\left(n\right)_{2}\left(n\right)_{1}}+\frac{\left(n\right)_{2}}{\left(n\right)_{1}\left(n\right)_{2}}\\
 & \:+\frac{\left(n\right)_{3}}{\left(n\right)_{2}\left(n\right)_{2}}+\frac{\left(n\right)_{3}}{\left(n\right)_{2}\left(n\right)_{2}}+\frac{\left(n\right)_{2}}{\left(n\right)_{2}\left(n\right)_{2}}+\frac{\left(n\right)_{1}}{\left(n\right)_{1}\left(n\right)_{1}}\\
= & \:\frac{n}{n-1}=1+\frac{1}{n}+O\left(\frac{1}{n^{2}}\right)
\end{align*}
This demonstrates Theorem \thmref{avg_fixed_points} and Table \tabref{Primitivity-Rank-and-fixed-points}
for $H=\left\langle x_{1}x_{2}x_{1}^{-1}x_{2}^{-1}\right\rangle $
(recall the discussion following Lemma \lemref{detecting_AE}, where
it is shown that $\pi\left(H\right)=2$ and that $\crit\left(H\right)=\left\{ \F_{2}\right\} $).

The explicit computation of $\Phi$ yields an effective version of
Theorem \nameref{thm:mp_is_prim_ext}: 
\begin{cor}
\label{cor:effective-bound-subgroups}Let $H\fg\F_{k}$, and let $\ell$
denote the number of edges in $\Gamma_{X}\left(H\right)$. Then $H\ff\F_{k}$
iff $\Phi_{H,\F_{k}}\left(n\right)=n^{-\rrk H}$ for $n\leq\ell+\rrk H$.
In particular, Proposition \ref{prop:effective-bound} follows.\end{cor}
\begin{proof}
Assume that $\Phi_{H,\F_{k}}\left(n\right)=n^{-\rrk H}$ holds for
$n\leq\ell+\rrk H$, and denote 
\begin{equation}
\Phi'\left(n\right)=\sum_{M\in\XF{H}}\frac{\left(n\right)_{\left|V\left(\Gamma_{X}\left(M\right)\right)\right|}}{\prod_{j=1}^{k}\left(n\right)_{\left|E_{j}\left(\Gamma_{X}\left(M\right)\right)\right|}},\label{eq:expression for M}
\end{equation}
where $E_{j}\left(\Gamma\right)$ are the $j$-edges in $\Gamma$.
By Corollary \ref{cor:Phi-formula}, $\Phi'\left(n\right)=\Phi_{H,\F_{k}}\left(n\right)$
for $n\geq n_{0}=\max_{j=1..k}\left|E_{j}\left(\Gamma_{X}\left(H\right)\right)\right|$,
and in particular $\Phi'\left(n\right)=n^{-\rrk H}$ for $n_{0}\leq n\leq\ell+\rrk H$.
We proceed to show that $\Phi'\left(n\right)\equiv n^{-\rrk H}$,
which implies $\Phi_{H,\F_{k}}\left(n\right)=n^{-\rrk H}$ for $n\geq n_{0}$.
The conclusion then follows by Theorem \nameref{thm:mp_is_prim_ext}
(which is proved in the next section).

The number of $j$-edges in every quotient of $\Gamma_{X}\left(H\right)$
is at most $E_{j}\left(\Gamma_{X}\left(H\right)\right)$, so that
$\Phi'\left(n\right)g\left(n\right)$ is a polynomial for $g\left(n\right)=\prod_{j=1}^{k}\left(n\right)_{\left|E_{j}\left(\Gamma_{X}\left(H\right)\right)\right|}$.
We would like to establish 
\begin{equation}
\Phi'\left(n\right)g\left(n\right)n^{\rrk\left(H\right)}\equiv g\left(n\right),\label{eq:f_and_g}
\end{equation}
and we note that $\deg g=\ell$, and $\deg\Phi'\leq\max_{M\in\XF{H}}-\rrk\left(M\right)\leq0$
follows from Claim \ref{cla:core-graphs-properties}(\ref{enu:euler})
(assuming $H\neq id$). Therefore, the degrees of both sides of (\ref{eq:f_and_g})
are at most $\ell+\rrk H$, and it suffices to show they agree at
$\ell+\rrk H+1=\ell+\rk H$ points. We already know that they agree
for $n_{0}\leq n\leq\ell+\rrk H$. For $0\leq n<n_{0}$ it is clear
that $g\left(n\right)=0$. It turns out that the l.h.s.~vanishes
as well for these values of $n$. Expanding the l.h.s.\ gives 
\begin{equation}
n^{\rrk H}\cdot\sum_{M\in\XF{H}}\left(n\right)_{\left|V\left(\Gamma_{X}\left(M\right)\right)\right|}\prod_{j=1}^{k}\left(n-\left|E_{j}\left(\Gamma_{X}\left(M\right)\right)\right|\right)_{\left|E_{j}\left(\Gamma_{X}\left(H\right)\right)\right|-\left|E_{j}\left(\Gamma_{X}\left(M\right)\right)\right|},\label{eq:goal-nefesh}
\end{equation}
and each term in the sum vanishes for $0\leq n<n_{0}$: Choose $1\leq j\leq k$
for which $\left|E_{j}\left(\Gamma_{X}\left(H\right)\right)\right|=n_{0}$.
For each $M\in\XF{H}$ either $\left|E_{j}\left(\Gamma_{X}\left(M\right)\right)\right|\leq n$,
in which case $\left(n-\left|E_{j}\left(\Gamma_{X}\left(M\right)\right)\right|\right)_{n_{0}-\left|E_{j}\left(\Gamma_{X}\left(M\right)\right)\right|}=0$,
or $\left|E_{j}\left(\Gamma_{X}\left(M\right)\right)\right|>n$; as
different $j$-edges must have different origins, the latter implies
that $\left|V\left(\Gamma_{X}\left(M\right)\right)\right|>n$, hence$\left(n\right)_{\left|V\left(\Gamma_{X}\left(M\right)\right)\right|}$
vanishes.\end{proof}
\begin{rem}
The discussion in this section suggests a generalization of our analysis
to finite groups $G$ other than $S_{n}$. For any (finite) faithful
$G$-set $S$, one can consider a random $|S|$-covering of $\G_{X}\left(J\right)$.
The fiber above every edge is chosen according to the action on $S$
of a (uniformly distributed) random element of $G$. In this more
general setting we also get a one-to-one correspondence between $\Hom\left(\F_{k},G\right)$
and $|S|$-coverings. Although the computation of $L^{X}$ and of
$\Phi$ might be more involved, this suggests a way of analyzing words
which are measure preserving w.r.t.\ $G$. 
\end{rem}

\section{\label{sec:proof}The proof of Theorem \thmref{avg_fixed_points}}

The last major ingredient of the proof of our main result, Theorem
\thmref{avg_fixed_points}, is an analysis of $C^{X}$, the double-sided
derivation of $\Phi$. Recall Definition \defref{distance} where
the $X$-distance $\rho_{X}\left(H,J\right)$ was defined for every
$H,J\fg\F_{k}$ with $H\covers J$. 
\begin{prop}
\label{prop:order-magnitude-C}Let $M,N\fg\F_{k}$ satisfy $M\covers N$.
Then 
\[
C_{M,N}^{X}\left(n\right)=O\left(\frac{1}{n^{\rrk(M)+\rho_{X}(M,N)}}\right)
\]

\end{prop}
Section \subref{C} is dedicated to the proof of this proposition.
Before getting there, we show how it practically finishes the proof
of our main result. We do this with the following final step:
\begin{prop}
\label{prop:order-magnitude-R}Let $H,N\fg\F_{k}$ satisfy $H\alg N$.
Then 
\[
R_{H,N}\left(n\right)=\frac{1}{n^{\rrk\left(N\right)}}+O\left(\frac{1}{n^{\rrk\left(N\right)+1}}\right)
\]
\end{prop}
\begin{proof}
Let $X$ be some basis of $\F_{k}$. Recall that $R=\zeta^{X}*C^{X}$,
i.e. 
\[
R_{H,N}\left(n\right)=\sum_{M\in\XC{H}{N}}C_{M,N}^{X}\left(n\right).
\]
For $M=N$ we have $C_{N,N}^{X}\left(n\right)=R_{N,N}\left(n\right)=\Phi_{N,N}\left(n\right)=n^{-\rrk\left(N\right)}$
(the last equality follows from the fact that $m$ independent uniform
permutations fix a point with probability $n^{-m}$). For any other
$M$, i.e.\ $M\in\XCO{H}{N}$, the fact that $N$ is an algebraic
extension of $H$ means that $M$ is \emph{not} a free factor of $N$
and therefore, by Theorem \thmref{rho} (and \eqref{rk-rk_le_rho}),
$\rho_{X}\left(M,N\right)\ge\rrk\left(N\right)-\rrk\left(M\right)+1$.
Proposition \propref{order-magnitude-C} then shows that 
\[
C_{M,N}^{X}\left(n\right)\in O\left(\frac{1}{n^{\rrk(M)+\rho_{X}(M,N)}}\right)\subseteq O\left(\frac{1}{n^{\rrk\left(N\right)+1}}\right).
\]
Hence,
\[
R_{H,N}\left(n\right)=C_{N,N}^{X}\left(n\right)+\sum_{M\in\XCO{H}{N}}C_{M,N}^{X}\left(n\right)=\frac{1}{n^{\rrk\left(N\right)}}+O\left(\frac{1}{n^{\rrk\left(N\right)+1}}\right).
\]

\end{proof}
The proof of Theorem \thmref{avg_fixed_points} is now at hand. For
every $H,J\fg\F_{k}$ with $H\le J$, by \eqref{phi=00003Dsum-of-alg}
and Proposition \propref{order-magnitude-R}, 
\begin{align*}
\Phi_{H,J}\left(n\right) & =\sum_{N\,:\, H\alg N\le J}R_{H,N}\left(n\right)\\
 & =R_{H,H}\left(n\right)+\sum_{N\,:\, H\algne N\le J}R_{H,N}\left(n\right)\\
 & =\frac{1}{n^{\rrk\left(H\right)}}+\sum_{N\,:\, H\algne N\le J}\frac{1}{n^{\rrk(N)}}+O\left(\frac{1}{n^{\rrk(N)+1}}\right).
\end{align*}
For $J=\F_{k}$ we can be more concrete. Recall that the $H$-critical
groups, $\crit\left(H\right)$, are the algebraic extensions of $H$
of minimal rank (other than $H$ itself), and this minimal rank is
$\pi\left(H\right)$. Therefore, 
\begin{align*}
\Phi_{H,\F_{k}}\left(n\right) & =\frac{1}{n^{\rrk(H)}}+\sum_{N\in\left(H,\infty\right)_{alg}}\frac{1}{n^{\rrk(N)}}+O\left(\frac{1}{n^{\rrk(N)+1}}\right)\\
 & =\frac{1}{n^{\rrk(H)}}+\frac{\left|\crit(H)\right|}{n^{\rp(H)}}+O\left(\frac{1}{n^{\rp(H)+1}}\right).
\end{align*}
This establishes our main results: Theorem \thmref{avg_fixed_points},
Theorem \thmref{mp_is_prim} and all their corollaries.

\subsection{\label{sub:C}The analysis of $C_{M,N}^{X}$}

In this subsection we look into $C^{X}$, the double-sided derivation
of $\Phi$, and establish Proposition \propref{order-magnitude-C},
which bounds the order of magnitude of $C_{M,N}^{X}$. Recall that
by definition $C^{X}=L^{X}*\mu^{X}$, which is equivalent to 
\begin{equation}
L_{M,J}^{X}=\sum_{N\in\XC{M}{J}}C_{M,N}^{X}\qquad\left(\forall M\covers J\right)\label{eq:L=00003Dsum_C}
\end{equation}
We derive a combinatorial meaning of $C_{M,N}^{X}$ from this relation.
To obtain this, we further analyze the rational expression \eqref{L-formula}
for $L_{M,J}^{X}$ and write it as a formal power series. Then, using
a combinatorial interpretation of the terms in this series, we attribute
each term to some $N\in\XC{M}{J}$, and show that for every $N\in\XC{M}{J}$,
the sum of terms attributed to $N$ is nothing but $C_{M,N}^{X}$.
Finally, we use this combinatorial interpretation of $C_{M,N}^{X}$
to estimate its order of magnitude.

\subsubsection*{Rewriting $L_{M,J}^{X}$ as a power series in $n^{-1}$}

Consider the numerator and denominator of \eqref{L-formula}: these
are products of expressions of the type $\left(n\right)_{r}$. It
is a classical fact that 
\[
\left(n\right)_{r}=\sum_{j=1}^{r}(-1)^{r-j}\left[{r\atop j}\right]n^{j}
\]
where $\left[{r\atop j}\right]$ is the \emph{unsigned} \emph{Stirling
number of the first kind.} That is, $\left[{r\atop j}\right]$ is
the number of permutations in $S_{r}$ with exactly $j$ cycles (see,
for instance, \cite{vLW01}, Chapter 13). 

We introduce the notation $\left[r\right]_{j}\overset{{\scriptscriptstyle def}}{=}\left[{r\atop r-j}\right]$,
which is better suited for our purposes. The cycles of a permutation
$\sigma\in S_{r}$ constitute a partition $P_{\sigma}$ of $\left\{ 1,\ldots,r\right\} $.
We define $\left\Vert \sigma\right\Vert =\left\Vert P_{\sigma}\right\Vert $
(recall \eqref{Partition_norm}), and it is immediate that $\left[r\right]_{j}$
is the number of permutations $\sigma\in S_{r}$ with $\left\Vert \sigma\right\Vert =j$.
It is also easy to see that $\left\Vert \sigma\right\Vert $ is the
minimal number of transpositions needed to be multiplied in order
to obtain $\sigma$. Therefore, $\left[r\right]_{j}$ is the number
of permutations in $S_{r}$ which can be expressed as a product of
$j$ transpositions, but no less. In terms of this notation, we obtain
\[
\left(n\right)_{r}=n^{r}\sum_{j=0}^{r-1}\left(-1\right)^{j}\left[r\right]_{j}n^{-j}.
\]
The product of several expressions of this form, namely $\left(n\right)_{r_{1}}\left(n\right)_{r_{2}}\ldots\left(n\right)_{r_{\ell}}$,
can be written as a polynomial in $n$ whose coefficients have a similar
combinatorial meaning, as follows. Let $X$ be a set, and $\varphi:X\rightarrow\left\{ 1,\ldots,\ell\right\} $
some function with fibers of sizes $\left|\varphi^{-1}\left(i\right)\right|=r_{i}$
($1\leq i\leq\ell$). We denote by 
\[
\Sym{}_{\varphi}\left(X\right)=\left\{ \sigma\in\Sym\left(X\right)\,\middle|\,\varphi\circ\sigma=\varphi\right\} 
\]
the set of permutations $\sigma\in\Sym\left(X\right)$ subordinate
to the partition of $X$ induced by the fibers of $\varphi$, i.e.,
such that $\varphi\left(\sigma\left(x\right)\right)=\varphi\left(x\right)$
for all $x\in X$. We define 
\[
\left[X\right]_{j}^{\varphi}=\left|\left\{ \sigma\in\Sym{}_{\varphi}\left(X\right)\,:\,\left\Vert \sigma\right\Vert =j\right\} \right|,
\]
the number of $\varphi$-subordinate permutations with $\left\Vert \sigma\right\Vert =j$.
Put differently, $\left[X\right]_{j}^{\varphi}$ counts the permutations
counted in $\left[\vphantom{\big|}\left|X\right|\vphantom{\big|}\right]_{j}$
which satisfy, in addition, that every cycle consists of a subset
of some fiber of $\varphi$. With this new notation, one can write:
\[
\left(n\right)_{r_{1}}\left(n\right)_{r_{2}}\ldots\left(n\right)_{r_{\ell}}=\prod_{i=1}^{l}\left(n^{r_{i}}\sum_{m=0}^{r_{i}-1}\left(-1\right)^{m}\left[r_{i}\right]_{m}n^{-m}\right)=n^{\left|X\right|}\sum_{j=0}^{\left|X\right|}\left(-1\right)^{j}\left[X\right]_{j}^{\varphi}n^{-j}
\]

Turning back to \eqref{L-formula}, we let $V_{M}$ and $E_{M}$ denote
the sets of vertices and edges, respectively, of $\G_{X}\left(M\right)$.
We denote by $\eta$ the morphism $\eta_{M\rightarrow J}^{X}$, and
use it implicitly also for its restrictions to $V_{M}$ and $E_{M}$,
which should cause no confusion. We obtain 
\[
L_{M,J}^{X}\left(n\right)=\frac{n^{|V_{M}|}\sum\limits _{j=0}^{|V_{M}|}\left(-1\right)^{j}\left[V_{M}\right]_{j}^{\eta}n^{-j}}{n^{|E_{M}|}\sum\limits _{j=0}^{|E_{M}|}\left(-1\right)^{j}\left[E_{M}\right]_{j}^{\eta}n^{-j}},
\]
which by Claim \claref{core-graphs-properties}\enuref{euler} equals
\begin{equation}
L_{M,J}^{X}\left(n\right)=n^{-\rrk\left(M\right)}\frac{\sum\limits _{j=0}^{|V_{M}|}\left(-1\right)^{j}\left[V_{M}\right]_{j}^{\eta}n^{-j}}{\sum\limits _{j=0}^{|E_{M}|}\left(-1\right)^{j}\left[E_{M}\right]_{j}^{\eta}n^{-j}}.\label{eq:L-half-way}
\end{equation}
Consider the denominator of \eqref{L-half-way} as a power series
$Q\left(n^{-1}\right)$. Its free coefficient is $\left[E_{M}\right]_{0}^{\eta}=1$.
This makes it relatively easy to get a formula for its inverse $1/Q\left(n^{-1}\right)$
as a power series. In general, if $Q\left(x\right)=1+\sum_{i=1}^{\infty}a_{i}x^{i}$,
then 
\begin{align*}
\frac{1}{Q(x)} & =\frac{1}{1-\sum_{i=1}^{\infty}(-a_{i})x^{i}}=\sum_{t=0}^{\infty}\left(\sum_{i=1}^{\infty}(-a_{i})x^{i}\right)^{t}=\\
 & =\sum_{t=0}^{\infty}\sum_{j_{1},j_{2},\ldots,j_{t}\ge1}(-1)^{t}a_{j_{1}}\cdot\ldots\cdot a_{j_{t}}x^{\sum_{i=1}^{t}j_{i}}.
\end{align*}
In the denominator of \eqref{L-half-way} we have $a_{i}=\left(-1\right)^{i}\left[E_{M}\right]_{i}^{\eta}$,
and the resulting expression needs to be multiplied with the numerator
$\sum_{j=0}^{|V_{M}|}\left(-1\right)^{j}\left[V_{M}\right]_{j}^{\eta}n^{-j}$.
In total, we obtain 
\begin{multline}
\begin{aligned}L_{M,J}^{X}\left(n\right) & =\\
 & \sum_{t=0}^{\infty}\sum_{{j_{0}\ge0\atop j_{1},\ldots,j_{t}\ge1}}\left(-1\right)^{t+\sum_{i=0}^{t}j_{i}}\left[V_{M}\right]_{j_{0}}^{\eta}\cdot\left[E_{M}\right]_{j_{1}}^{\eta}\cdot\ldots\cdot\left[E_{M}\right]_{j_{t}}^{\eta}n^{-\rrk\left(M\right)-\sum_{i=0}^{t}j_{i}}.
\end{aligned}
\label{eq:huge-sum}
\end{multline}

\subsubsection*{The combinatorial meaning and order of magnitude of $C_{M,N}^{X}$}

The expression \eqref{huge-sum} is a bit complicated, but it presents
$L_{M,J}^{X}\left(n\right)$ as a sum (with coefficients $\pm n^{-s}$)
of terms with a combinatorial interpretation: the term $\left[V_{M}\right]_{j_{0}}^{\eta}\cdot\left[E_{M}\right]_{j_{1}}^{\eta}\cdot\ldots\cdot\left[E_{M}\right]_{j_{t}}^{\eta}$
counts $\left(t+1\right)$-tuples of $\eta$-subordinate permutations.
The crux of the matter is that this interpretation allows us to attribute
each tuple to a specific subgroup $N\in\XC{M}{J}$. This is done as
follows.

Let $\left(\sigma_{0},\sigma_{1},\ldots,\sigma_{t}\right)$ be a $\left(t+1\right)$-tuple
of permutations such that $\sigma_{0}\in\mathrm{\Sym}_{\,\eta}\left(V_{M}\right)$
and $\sigma_{1},\ldots,\sigma_{t}\in\mathrm{\Sym}_{\,\eta}\left(E_{M}\right)\backslash\left\{ \mathrm{id}\right\} $
(we exclude $\mathrm{id}\in\Sym\left(E_{M}\right)$, which is the
only permutation counted in $\left[E_{M}\right]_{0}^{\eta}$). Consider
the graph $\Gamma=\nicefrac{\Gamma_{X}\left(M\right)}{\left\langle \sigma_{0},\ldots,\sigma_{t}\right\rangle }$,
which is the quotient of $\Gamma_{X}\left(M\right)$ by all identifications
of pairs of the form $v,\sigma_{0}\left(v\right)$ ($v\in V_{M}$)
and $e,\sigma_{i}\left(e\right)$ ($e\in E_{M}$, $1\leq i\leq t$)%
\footnote{For the definition of the quotient of a graph by identifications of
vertices see the discussion preceding Figure \ref{fig:quotient-graph}.
Although we did not deal with merging of edges before, this is very
similar to merging vertices. Identifying a pair of edges means identifying
the pair of origins, the pair of termini and the pair of edges. In
terms of the generated core graph (see Section \ref{sec:Core-Graphs}),
identifying a pair of edges is equivalent to identifying the pair
of origins and\textbackslash{}or the pair of termini.%
}. Since $\Gamma$ is obtained from $\Gamma_{X}\left(M\right)$ by
identification of elements with the same $\eta$-image, $\eta$ induces
a well defined morphism $\Gamma\rightarrow\Gamma_{X}\left(J\right)$.
Thus, every closed path in $\Gamma$ projects to a path in $\Gamma_{X}\left(J\right)$,
giving $\pi_{1}^{X}\left(\Gamma\right)\leq\pi_{1}^{X}\left(\Gamma_{X}\left(J\right)\right)=J$.
We denote $N=N_{\sigma_{0},\sigma_{1},\ldots,\sigma_{t}}=\pi_{1}^{X}\left(\Gamma\right)$.
As usual (see Figures \figref{folding_process}, \figref{quotient-graph}),
we can perform Stallings foldings on $\Gamma$ until we obtain the
core graph corresponding to $N$, $\Gamma_{X}\left(N\right)$. Obviously
we have $M\covers N$, and by Claim \claref{cover-properties}\enuref{Cover-insideCover}
also $N\covers J$. Thus, we always have $N=N_{\sigma_{0},\sigma_{1},\ldots,\sigma_{t}}\in\XC{M}{J}$.
To summarize the situation:
\begin{equation}
\xymatrix@C=35pt{\Gamma_{X}\left(M\right)\ar@{->>}[r]\ar@/_{1pc}/@{->>}[rr]_{\eta_{M\rightarrow N}^{X}} & \Gamma=\nicefrac{\Gamma_{X}\left(M\right)}{\left\langle \sigma_{0},\ldots,\sigma_{1}\right\rangle }\ar@{->>}[r]^{\qquad{\scriptscriptstyle folding}} & \Gamma_{X}\left(N\right)\ar@{->>}[r]^{\eta_{N\rightarrow J}^{X}} & \Gamma_{X}\left(J\right)}
\label{eq:quotient_by_tuple}
\end{equation}

Our next move is to rearrange \eqref{huge-sum} according to the intermediate
subgroups $N\in\XC{M}{J}$ which correspond to the tuples counted
in it. For any $N\in\XC{M}{J}$ we denote by $\mathcal{T}_{M,N,J}^{X}$
the set of tuples $\left(\sigma_{0},\sigma_{1},\ldots,\sigma_{t}\right)$
such that $N_{\sigma_{0},\sigma_{1},\ldots,\sigma_{t}}=N$, i.e.
\begin{align*}
\mathcal{T}_{M,N,J}^{X} & =\left\{ \left(\sigma_{0},\sigma_{1},\ldots,\sigma_{t}\right)\,\middle|\,\begin{matrix}t\in\mathbb{N},\:\sigma_{0}\in\mathrm{\Sym}_{\,\eta}\left(V_{M}\right)\\
\sigma_{1},\ldots,\sigma_{t}\in\Sym_{\,\eta}\left(E_{M}\right)\backslash\left\{ \mathrm{id}\right\} \vphantom{\Big|}\\
\pi_{1}^{X}\left(\nicefrac{\Gamma_{X}\left(M\right)}{\left\langle \sigma_{0},\sigma_{1},\ldots,\sigma_{t}\right\rangle }\right)=N
\end{matrix}\right\} .
\end{align*}
The terms in \eqref{huge-sum} which correspond to a fixed $N\in\XC{M}{J}$
thus sum to
\begin{equation}
\widetilde{C}_{M,J}^{X}\left(N\right)=\sum_{\left(\sigma_{0},\sigma_{1},\ldots,\sigma_{t}\right)\in\mathcal{T}_{M,N,J}^{X}}\frac{\:\:\left(-1\right)^{t+\sum\limits _{i=0}^{t}\left\Vert \sigma_{i}\right\Vert }\:\:}{n^{\rrk\left(M\right)+\sum\limits _{i=0}^{t}\left\Vert \sigma_{i}\right\Vert }},\label{eq:C_MN_formula}
\end{equation}
and \eqref{huge-sum} becomes 
\begin{equation}
L_{M,J}^{X}=\sum_{N\in\XC{M}{J}}\widetilde{C}_{M,J}^{X}\left(N\right)\label{eq:D_MJ_N}
\end{equation}
The equation \eqref{D_MJ_N} looks much like \eqref{L=00003Dsum_C},
with $\widetilde{C}_{M,J}^{X}\left(N\right)$ playing the role of
$C_{M,N}^{X}$. In order to establish equality between the latter
two, we must show that $\widetilde{C}_{M,J}^{X}\left(N\right)$ does
not depend on $J$. Fortunately, this is not hard: it turns out that
\begin{equation}
\qquad\widetilde{C}_{M,J}^{X}\left(N\right)=\widetilde{C}_{M,N}^{X}\left(N\right)\qquad\left(\forall N\in\XC{M}{J}\right),\label{eq:C_MJN_equals_C_MNN}
\end{equation}
and the r.h.s.\ is, of course, independent of $J$. This equality
follows from $\mathcal{T}_{M,N,J}^{X}=\mathcal{T}_{M,N,N}^{X}$, which
we now justify. The only appearance $J$ makes in the definition of
$\mathcal{T}_{M,N,J}^{X}$ is inside $\eta=\eta_{M\rightarrow J}^{X}$,
which is to be $\sigma_{i}$-invariant (for $0\leq i\leq n$), i.e.,
$\sigma_{i}$ must satisfy $\eta_{M\rightarrow J}^{X}\circ\sigma_{i}=\eta_{M\rightarrow J}^{X}$.
If $\left(\sigma_{0},\ldots,\sigma_{t}\right)\in\mathcal{T}_{M,N,J}^{X}$
then $\eta_{M\rightarrow N}^{X}\circ\sigma_{i}=\eta_{M\rightarrow N}^{X}$
follows from the fact that $\Gamma_{X}\left(N\right)$ is a quotient
of $\nicefrac{\Gamma_{X}\left(M\right)}{\left\langle \sigma_{i}\right\rangle }$.
On the other hand, if $\left(\sigma_{0},\ldots,\sigma_{t}\right)\in\mathcal{T}_{M,N,N}^{X}$
then we have $\eta_{M\rightarrow N}^{X}\circ\sigma_{i}=\eta_{M\rightarrow N}^{X}$,
hence also (see \eqref{quotient_by_tuple}) 
\[
\eta_{M\rightarrow J}^{X}\circ\sigma_{i}=\eta_{N\rightarrow J}^{X}\circ\eta_{M\rightarrow N}^{X}\circ\sigma_{i}=\eta_{N\rightarrow J}^{X}\circ\eta_{M\rightarrow N}^{X}=\eta_{M\rightarrow J}^{X}.
\]

Writing $\widetilde{C}_{M,N}^{X}\overset{{\scriptscriptstyle def}}{=}\widetilde{C}_{M,N}^{X}\left(N\right)$,
we have by \eqref{L=00003Dsum_C}, \eqref{D_MJ_N}, and \eqref{C_MJN_equals_C_MNN}
\[
C^{X}*\zeta^{X}=L^{X}=\widetilde{C}^{X}*\zeta^{X}
\]
which shows that $C^{X}=\widetilde{C}^{X}$, as desired.

We approach the endgame. Let $\left(\sigma_{0},\sigma_{1},\ldots,\sigma_{t}\right)\in\mathcal{T}_{M,N,J}^{X}=\mathcal{T}_{M,N,N}^{X}$,
and consider the partition $P$ of $V\left(\Gamma_{X}\left(H\right)\right)$,
obtained by identifying $v$ and $v'$ whenever $\sigma_{0}\left(v\right)=v'$,
or $\sigma_{i}\left(e\right)=e'$ for some $1\leq i\leq t$ and edges
$e,e'$ whose origins are $v$ and $v'$, respectively. Since $P$
can clearly be obtained by $\sum_{i=0}^{t}\left\Vert \sigma_{i}\right\Vert $
identifications, we have $\left\Vert P\right\Vert \leq\sum_{i=0}^{t}\left\Vert \sigma_{i}\right\Vert $
(a strong inequality can take place - for example, one can have $\sigma_{1}=\sigma_{2}$).
Since $\left(\sigma_{0},\sigma_{1},\ldots,\sigma_{t}\right)\in\mathcal{T}_{M,N,J}^{X}$
we have $\pi_{1}^{X}\left(\nicefrac{\Gamma_{X}\left(H\right)}{P}\right)=N$,
and thus by \eqref{rho_from_partition} we obtain
\[
\rho_{X}\left(H,J\right)\leq\left\Vert P\right\Vert \leq\sum_{i=0}^{t}\left\Vert \sigma_{i}\right\Vert .
\]
From \eqref{C_MN_formula} (recall that $\widetilde{C}_{M,J}^{X}\left(N\right)=\widetilde{C}_{M,N}^{X}=C_{M,N}^{X}$)
we now have 

\[
C_{M,N}^{X}\left(n\right)=O\left(\frac{1}{n^{\rrk(M)+\rho_{X}(M,N)}}\right),
\]
and Proposition \propref{order-magnitude-C} is proven.

\section{\label{sec:Profinite}Primitive words in the profinite topology}

Theorem \thmref{mp_is_prim} has some interesting implications to
the study of profinite groups. In fact, some of the original interest
in the conjecture that is proven in this paper stems from these implications.

Let $\hF$ denote the profinite completion of the free group $\F_{k}$.
A basis of $\hF$ is a set $S\subset\hF$ such that every map from
$S$ to a profinite group $G$ admits a unique extension to a continuous
homomorphism $\hF\to G$. It is a standard fact that $\F_{k}$ is
embedded in $\hF$, and that every basis of $\F_{k}$ is also a basis
of $\hF$ (see for example \cite{Wil98}). An element of $\hF$ is
called \emph{primitive }if it belongs to a basis of $\hF$. 

It is natural to ask whether an element of $\F_{k}$, which is primitive
in $\hF$, is already primitive in $\F_{k}$. In fact, this was conjectured
by Gelander and by Lubotzky, independently. Theorem \thmref{mp_is_prim}
yields a positive answer, as follows. An element $w\in\hF$ is said
to be \emph{measure preserving} if for any finite group $G$, and
a uniformly distributed random (continuous) homomorphism $\hat{\alpha}_{G}\in\Hom_{cont}\left(\hF,G\right)$,
the image $\hat{\alpha}_{G}\left(w\right)$ is uniformly distributed
in $G$. By the natural correspondence $\Hom_{cont}\left(\hF,G\right)\cong\Hom\left(\F_{k},G\right)$,
an element of $\F_{k}$ is measure preserving w.r.t.\ $\F_{k}$ iff
it is so w.r.t.\ $\hF$. As in $\F_{k}$, a primitive element of
$\hF$ is easily seen to be measure preserving. Theorem \thmref{mp_is_prim}
therefore implies that if $w\in\F_{k}$ is primitive in $\hF$, then
it is also primitive in $\F_{k}$. In other words:
\begin{cor}
\label{cor:prim_in_F_k_iff_prim_in_hF_k} Let $P$ denote the set
of primitive elements of $\F_{k}$, and let $\widehat{P}$ denote
the set of primitive elements of $\hF.$ Then 
\[
P=\widehat{P}\cap\F_{k}.
\]

\end{cor}
As $\widehat{P}$ is a closed set in $\hF$, this immediately implies
Corollary \corref{prim_are_closed}, which states that $P$ is closed
in the profinite topology. In fact, there is also a direct proof to
Corollary \corref{prim_are_closed} from Theorem \thmref{avg_fixed_points}:
one has to find, for every non-primitive word $w\in\F_{k}$, some
$H\leq_{\mathrm{f.i.}}\F_{k}$ such that the coset $wH$ contains
no primitives. By Theorem \thmref{avg_fixed_points} there exists
$n$ so that $w$ does not induce uniform distribution on $S_{n}$.
For this $n$, let 
\[
H=\bigcap_{\alpha:\F_{k}\to S_{n}}\ker\alpha
\]
and then $wH$ is a primitive-free coset (as all words in the same
coset of $H$ induce the exact same measure on $S_{n}$).

\medskip{}

This circle of ideas has a natural generalization. Observe the following
five equivalence relations on the elements of $\F_{k}$:
\begin{itemize}
\item $w_{1}\overset{A}{\sim}w_{2}$ if $w_{1}$ and $w_{2}$ belong to
the same $\Aut\F_{k}$-orbit.
\item $w_{1}\overset{B}{\sim}w_{2}$ if $w_{1}$ and $w_{2}$ belong to
the same $\overline{\Aut\F_{k}}$-orbit (where $\overline{\Aut\F_{k}}$
is the closure of $\Aut\F_{k}$ in $\Aut\hF$).
\item $w_{1}\overset{C}{\sim}w_{2}$ if $w_{1}$ and $w_{2}$ belong to
the same $\Aut\hF$-orbit.
\item $w_{1}\overset{C'}{\sim}w_{2}$ if $w_{1}$ and $w_{2}$ have the
same ``statistical'' properties, namely if they induce the same
distribution on any finite group.
\item $w_{1}\overset{C''}{\sim}w_{2}$ if the evaluation maps $ev_{w_{1}},ev_{w_{2}}:\mathrm{Epi}\left(F_{k},G\right)\rightarrow G$
have the same images for every finite group $G$.
\end{itemize}
It is not hard to see that $\left(A\right)\Rightarrow\left(B\right)\Rightarrow\left(C\right)\Rightarrow\left(C'\right)\Rightarrow\left(C''\right)$
(namely, that if $w_{1}\overset{A}{\sim}w_{2}$ then $w_{1}\overset{B}{\sim}w_{2}$,
and so on). The only nontrivial implication is $\left(C'\right)\Rightarrow\left(C''\right)$,
which can be shown by induction on the size of $G$. In an unpublished
manuscript, C.\ Meiri gave a one-page proof that $\left(C\right)$,
$\left(C'\right)$ and $\left(C''\right)$ in fact coincide (in fact,
these three coincide for all elements of $\hF$). 

From this perspective, our main result shows that in the case that
$w_{1}$ is primitive, all five relations coincide, and it is natural
to conjecture that they in fact coincide for all elements in $\F_{k}$%
\footnote{In \cite{AV10}, for example, the authors indeed ask whether $\left(C'\right)\Rightarrow\left(A\right)$. %
}. Showing that $\left(A\right)\Leftarrow\left(B\right)$ would imply
that $\Aut\F_{k}$-orbits in $\F_{k}$ are closed in the profinite
topology, and the stronger statement $\left(A\right)\Leftarrow\left(C\right)$
would imply that words which lie in different $\Aut\F_{k}$-orbits
can be told apart using statistical methods. 

The analysis which is carried out in this paper does not suffice for
the general case. For example, consider the words $w_{1}=x_{1}x_{2}x_{1}x_{2}^{\,-1}$
and $w_{2}=x_{1}x_{2}x_{1}^{\,-1}x_{2}^{\,-1}$. They belong to different
$\Aut\F_{2}$-orbits, as $w_{2}\in\F_{2}'$ but $w_{1}\notin\F_{2}'$,
but induce the same distribution on $S_{n}$ for every $n$: their
images under a random homomorphism are a product of a random permutation
($\sigma$) and a random element in its conjugacy class ($\tau\sigma\tau^{-1}$
for $w_{1}$, and $\tau\sigma^{-1}\tau^{-1}$$ $ for $w_{2}$). However,
while $S_{n}$ do not distinguish between these two words, other groups
do (in fact these words induce the same distribution on $G$ precisely
when every element in $G$ is conjugate to its inverse, see \cite{parzanchevski2012fourier}
for a discussion of this).

\medskip{}

These questions also play a role in the theory of decidability in
infinite groups. A natural extension of the word-problem and the conjugacy-problem,
is the following \emph{automorphism-problem}: given a group $G$ generated
by $S$, and two words $w_{1},w_{2}\in F\left(S\right)$, can it be
decided whether $w_{1}$ and $w_{2}$ belong to the same $\Aut G$-orbit
in $G$? Whitehead's algorithm \cite{Whi36a,Whi36b} gives a concrete
solution when $G=\F_{k}$. Showing that $\left(A\right)\Leftarrow\left(B\right)$
would provide an alternative decision procedure for $\F_{k}$.

More generally, and in a similar fashion to the conjugacy problem,
it can be shown that if
\begin{enumerate}
\item $G$ is finitely presented
\item $\Aut G$ is finitely generated
\item $\Aut G$-orbits are closed in the profinite topology 
\end{enumerate}
then the automorphism-problem in $G$ is decidable. For the free group
$\left(1\right)$ and $\left(2\right)$ are known, and $\left(3\right)$
is exactly the conjectured coincidence $\left(A\right)\Leftrightarrow\left(B\right)$.

\section{Open questions}

We mention some open problems that naturally arise from the discussion
in this paper. 
\begin{itemize}
\item Section \ref{sec:Profinite} shows how the questions about primitive
elements can be extended to all $\Aut\F_{k}$-orbits in $\F_{k}$
(is it true that $\left(A\right)\Leftrightarrow\left(B\right)$, and
even the stronger equivalence $\left(A\right)\Leftrightarrow\left(C\right)$?).
More generally, can statistical properties tell apart two subgroups
$H_{1},H_{2}\fg\F_{k}$ which belong to distinct $\Aut\F_{k}$-orbits?
This would be a further generalization of Theorem \thmref{mp_is_prim}. 
\item It is also interesting to consider words which are measure preserving
w.r.t.\ other types of groups. For instance, does Theorem \thmref{mp_is_prim}
still hold if we replace ``finite groups'' by ``compact Lie groups'',
and study Haar-measure preserving words? Is there a single compact
Lie group which suffices? Within finite groups, we showed that measure
preservation w.r.t.\ $S_{n}$ implies primitivity. Is it still true
if we replace $S_{n}$ by some other infinite family of finite groups
(e.g.\ $\mathrm{PSL}_{n}\left(q\right)$, or solvable groups)?
\item Is it true that 
\[
\left[H,\infty\right)_{\leq}=\bigcup_{{X\mathrm{\, is\, a}\atop \mathrm{basis\, of\,}\F_{k}}}\XF{H}
\]
and under which assumptions does the following hold
\[
\AE{H}=\bigcap_{{X\mathrm{\, is\, a}\atop \mathrm{basis\, of\,}\F_{k}}}\XF{H}
\]
(see Remark \ref{rem:alg-cover-inc})?
\item The distribution induced by $w$ on a finite group $G$ is a class
function, and so is a linear combination of the characters of $G$
(for more on this point of view e.g.\ \cite{AV10,parzanchevski2012fourier}).
In particular, $\Phi_{\left\langle w\right\rangle ,\mathbf{F}_{k}}\left(n\right)-1$
is the coefficient of the \emph{standard} character of $S_{n}$. The
first nonzero term of $\Phi_{\left\langle w\right\rangle ,\mathbf{F}_{k}}-1$
encodes the primitivity rank and number of critical subgroups of $w$.
Can the next terms be given an algebraic interpretation, and can they
be estimated? (Such an estimation may contribute further to the study
of expansion in graphs, which started in \cite{Pud14b}.) What about
the coefficients of other characters of $S_{n}$ or of any other (family
of) groups?
\end{itemize}

\section*{Acknowledgments}

It is a pleasure to thank our advisors Nati Linial and Alex Lubotzky
for their support, encouragement and useful comments. We are also
grateful to Aner Shalev for supporting this research and for his valuable
suggestions. We would also like to thank Uri Bader, Tsachik Gelander,
Chen Meiri, Paul Nelson and Iddo Samet for their beneficial comments.
We have benefited much from the mathematical open source community,
and in particular from GAP \cite{GAP4}, and its free group algorithms
package \cite{FGA}.

\section*{Glossary}

\begin{center}
\begin{tabular}{|>{\centering}m{0.13\columnwidth}|>{\centering}m{0.32\columnwidth}|>{\centering}m{0.18\columnwidth}|>{\centering}m{0.3\columnwidth}|}
\hline 
 &  & Reference & Remarks\tabularnewline[\doublerulesep]
\hline 
$H\fg\F_{k}$ & finitely generated &  & \tabularnewline[\doublerulesep]
\hline 
$H\ff J$ & free factor &  & \tabularnewline[\doublerulesep]
\hline 
$H\leq_{alg}J$ & algebraic extension & Definition \defref{algebraic-extension} & \tabularnewline[\doublerulesep]
\hline 
$H\covers J$ & $H$ $X$-covers $J$ & Definition \defref{quotient} & $H\overset{X}{\twoheadrightarrow}J$ in \cite{Pud14a}\tabularnewline[\doublerulesep]
\hline 
$\mathfrak{sub}_{f\! g}\left(\mathbf{F}_{k}\right)$ & the set of finitely generated subgroups of $\F_{k}$ &  & \tabularnewline[\doublerulesep]
\hline 
$\left[H,J\right]_{\preceq}$ & $\left\{ L\,\middle|\, H\preceq L\preceq J\right\} $ &  & \multirow{3}{0.3\columnwidth}{\centering{}$\preceq$ is either one of $\leq,\ff,\leq_{alg}$ or
$_{\Xcov}$ (standing for $\covers$ )}\tabularnewline[\doublerulesep]
\cline{1-3} 
$\left[H,J\right)_{\preceq}$ & $\left\{ L\,\middle|\, H\preceq L\precneqq J\right\} $ &  & \tabularnewline[\doublerulesep]
\cline{1-3} 
$\left[H,\infty\right)_{\preceq}$ & $\left\{ L\,\middle|\, H\preceq L\right\} $ &  & \tabularnewline[\doublerulesep]
\hline 
$\XF{H}$ & the $X$-quotients of $H$ &  & $\mathcal{O}_{X}\left(H\right)$, or $X$-frigne in \cite{MVW07}\tabularnewline[\doublerulesep]
\hline 
$\AE{H}$ & algebraic extensions of $H$ &  & $AE\left(H\right)$ in \cite{MVW07}\tabularnewline[\doublerulesep]
\hline 
$\pi\left(H\right)$ & primitivity rank of $H$ & \multirow{2}{0.18\columnwidth}{Definition \defref{prim_rank}} & $\rp\left(H\right)=\pi\left(H\right)-1$\tabularnewline[\doublerulesep]
\cline{1-2} \cline{4-4} 
$\mathrm{Crit}\left(H\right)$ & $H$-critical groups &  & \tabularnewline[\doublerulesep]
\hline 
$\Gamma_{X}\left(H\right)$ & $X$-labeled core graph of $H$ &  & \tabularnewline[\doublerulesep]
\hline 
$\rho_{X}\left(H,J\right)$ & $X$-distance & Definition \defref{distance} & $H\covers J$\tabularnewline[\doublerulesep]
\hline 
$\eta_{H\rightarrow J}^{X}$  & the morphism $\Gamma_{X}\left(H\right)\rightarrow\Gamma_{X}\left(J\right)$ & Claim \claref{morphism-properties} & $H\leq J$\tabularnewline[\doublerulesep]
\hline 
$\alpha_{J,n}$ & a uniformly chosen random homomorphism in $\mathrm{Hom}\left(J,S_{n}\right)$ &  & $J\fg\F_{k}$\tabularnewline[\doublerulesep]
\hline 
$\Phi_{H,J}\left(n\right)$ & the expected number of common fixed points of $\alpha_{J,n}\left(H\right)$ & (\ref{eq:Phi}) & $H\le J$\tabularnewline[\doublerulesep]
\hline 
\end{tabular}
\par\end{center}

\bibliographystyle{amsalpha}
\bibliography{PrimitiveBib,pudra}

\end{document}